\documentclass[11pt]{amsart}
\usepackage{amsfonts, amsmath, amssymb, amscd, amsthm, graphicx,enumerate}

\usepackage{epsfig, xcolor}
\usepackage{hyperref} % Use [pagebackref] option to show on which pages eash bibitem is referenced
\hypersetup{linktocpage}

\hypersetup{colorlinks,
    linkcolor={red!50!black},
    citecolor={blue!80!black},
    urlcolor={blue!80!black}}
\usepackage{float}

\usepackage[capitalise,noabbrev]{cleveref}

\newtheorem{thm}{Theorem}[section]
\newtheorem{cor}[thm]{Corollary}
\newtheorem{lem}[thm]{Lemma}
\newtheorem{prop}[thm]{Proposition}

\newtheorem{ques}[thm]{Question}

\newtheorem{claim}[thm]{Claim}

\theoremstyle{definition}
\newtheorem{defn}[thm]{Definition}

\theoremstyle{remark}
\newtheorem{rem}[thm]{Remark}
\newtheorem{ex}[thm]{Example}

\newfont{\eufm}{eufm10}

\newcommand{\Prob}{\operatorname{Prob}}

\newcommand{\e}{\varepsilon }
\renewcommand{\phi}{\varphi}

\newcommand{\N}{\mathbb N}
\newcommand{\Z}{\mathbb Z}
\newcommand{\R}{\mathbb R}
\renewcommand{\ll }{\langle\hspace{-.7mm}\langle }
\newcommand{\rr }{\rangle\hspace{-.7mm}\rangle }

\renewcommand{\d }{{\rm d} }

\newcommand{\cone}{\mathrm{Cone}}
\def\mc {\mathcal}

\newcommand{\NN}{\mathbb{N}}
\newcommand{\ZZ}{\mathbb{Z}}

\renewcommand{\d}{{\rm d}}

\renewcommand{\d}{{\rm d}}

\newcommand{\Sub}{{\rm Sub}}

\usepackage{mathtools}

\newcounter{desccount}

\newcommand{\descref}[1]{\hyperref[#1]{#1}}

\newcommand{\apt}{apt} %Could replace this with "benign" or something similar.

\begin{document}

\title{Subgroup mixing and random walks in groups acting on hyperbolic spaces}
\author{M. Hull}
\author {A. Minasyan}
\author{D. Osin}
\thanks{The work of the third author has been supported by the NSF grants DMS-1612473, DMS-1853989, DMS-2405032.}

\address[Michael Hull]{Department of Mathematics \& Statistics,
University of North Carolina at Greensboro, 
116 Petty Building,
Greensboro, NC 27402, U.S.A.}
\email{mbhull@uncg.edu}
\address[Ashot Minasyan]{CGTA, School of Mathematical Sciences,
University of Southampton, Highfield, Southampton, SO17~1BJ, United Kingdom.}
\email{aminasyan@gmail.com}
\address[Denis Osin]{Department of Mathematics, Vanderbilt University, Nashville, TN 37240, U.S.A.}
\email{denis.osin@gmail.com}

\begin{abstract}
We study the topological dynamics of the action of an acylindrically hyperbolic group on the space of its infinite index convex cocompact subgroups by conjugation. We show that, for any suitable probability measure $\mu$, random walks with respect to $\mu$ will produce elements with strong mixing properties for this action asymptotically almost surely. In particular, when the group has no finite normal subgroups this implies that the action is highly topologically transitive. Along the way, we prove technical results about convex cocompact subgroups which allow us to extend some results on random walks of Abbott and the first author.
\end{abstract}

\maketitle

\section{Introduction}
Throughout this paper, all groups are assumed to be  countable and discrete. For a group $G$, let $\Sub(G)$ denote the set of subgroups of $G$. The product topology on $2^G$ induces the structure of a compact Polish space on $\Sub(G)$ and the rule \[H\mapsto gHg^{-1},\;\;\; \forall \,g\in G,\;\; \forall\, H\leqslant G,\] defines an action of the group $G$ on $\Sub(G)$ by homeomorphisms. Our main goal is to study the mixing properties of this action for groups acting on hyperbolic spaces. Our work is partially motivated by model-theoretic applications discussed in the forthcoming paper \cite{HMO}. We begin by recalling the necessary definitions from topological dynamics.

\subsection{Topological transitivity and \texorpdfstring{$\mu$-mixing}{mu-mixing}} 
Let $G$ be a group acting continuously on a topological space $X$. For any subsets $U$ and $V$ of $X$, we define 
\begin{equation}\label{Eq:TT}
N(U,V)=\{ g\in G\mid g(U)\cap V\ne \emptyset\}.
\end{equation}
The action $G\curvearrowright X$ is said to be \emph{topologically transitive} if $N(U,V)\ne \emptyset$  for any non-empty open sets $U,V\subseteq X$.  

%If $X$ is a Polish space, topological transitivity of the action $G\curvearrowright X$ is equivalent to the existence of a dense orbit (see Lemma~\ref{lem:tt_dense_orbit} below).

In this paper, we introduce a class of actions satisfying a stronger condition. More precisely, let $\Prob(G)$ denote the set of all (countably additive) probability measures on a group $G$. For $\mu\in \Prob(G)$, we denote by $\mu^{\ast n}$ the $n$-fold convolution of $\mu$ with itself. 

\begin{defn}\label{def:mu-mixing}
We say that an action of a group $G$ on a topological space $X$ is \emph{topologically $\mu$-mixing} for some $\mu\in \Prob(G)$ if $$\lim\limits_{n\to \infty} \mu^{\ast n}(N(U,V))=1$$ for any non-empty open subsets $U$ and $V$ of $X$. 
\end{defn}

In other words, the action of a group $G$ on a topological space $X$ is topologically $\mu$-mixing if the $n$-step random walk on $G$ generated by $\mu$ and starting at $1$ produces an element of $N(U,V)$ almost surely as $n\to \infty$. Clearly, every such action is topologically transitive. Our interest in topological $\mu$-mixing is partly driven by the observation that it implies a much stronger condition. Recall that an action of a group $G$  on a topological space $X$ is \emph{highly topologically transitive} if the diagonal action of $G$ on $X^k$ is topologically transitive for all $k\in \NN$. The following elementary fact is proved in Section~\ref{subsec:top_dyn}. 

\begin{prop}\label{Prop:HT} 
Let $G$ be a group. For any $\mu \in \Prob(G)$,  every topologically $\mu$-mixing action of $G$ on a topological space is highly topologically transitive.
\end{prop}

Clearly, a highly topologically transitive action of a group $G$ on a topological space $X$ does not need to be $\mu$-mixing in general. Indeed, if $X$ contains two disjoint non-empty open sets, no $G$-action on $X$ is $\mu$-mixing with respect to the Dirac measure $\mu=\delta_1$ concentrated at $1$. Furthermore, in Proposition~\ref{prop:high_top_act_not_mixing}, we give an example of a group $G$, generated by a finite symmetric set $A$, and a highly topologically transitive action of $G$ on a Cantor space that is not $\mu$-mixing for the uniform measure $\mu$ supported on $A$ (as usual, we say that a subset $A$ of a group $G$ is \textit{symmetric} if $A^{-1}=A$.) Thus, $\mu $-mixing is a stronger property than high topological transitivity even for well-behaved measures. 

Generally, the conjugation action of a group $G$ on  $\Sub (G)$ is not topologically transitive. One obstacle is that if $G$ is finitely generated, then every finite index subgroup is an isolated point of $\Sub (G)$. However, even if we restrict to the subspace 
\[
\Sub_\infty (G)=\{ H\leqslant G \mid |G:H|=\infty\},
\] 
topological transitivity remains a rare phenomenon. For example, we show that the action of an infinite virtually solvable group $G$ on $\Sub_{\infty} (G)$ by conjugation is topologically transitive if and only if $G\cong \ZZ$ (see Corollary~\ref{cor:solv}). 

On the other hand, it is not difficult to prove that the action of the free group $F_n$ of rank $n\ge 1$ on $\Sub_{\infty}(F_n)$ is topologically transitive.  In this paper, we generalize and strengthen the latter observation in two directions. First, we replace the free group with a wide class of groups acting on hyperbolic spaces. Second, we prove that the action of such groups on a natural analogue of $\Sub_{\infty}(F_n)$ is topologically $\mu$-mixing (and, in particular, highly topologically transitive).  To formulate our main result we need to review additional terminology.

\subsection{Groups acting on hyperbolic spaces} All group actions on hyperbolic metric spaces considered in this paper are assumed to be by isometries, unless specified otherwise. Let $G$ be a group acting on a hyperbolic geodesic metric space $S$.

An element $g\in G$ is called \emph{loxodromic} if the map $\Z\to S$, $n\mapsto g^n(s)$, is a quasi-isometric embedding. We will say that the action of $G$ on $S$ is  \emph{partially WPD} if $G$ contains at least one loxodromic WPD element.
Here ``WPD" stands for the weak proper discontinuity condition introduced by Bestvina and Fujiwara in \cite{BF} (see Section~\ref{sec:back} for a definition). Note that a non-(virtually cyclic) group $G$ admits a partially WPD action on a hyperbolic metric space if and only if $G$ is acylindrically hyperbolic, see \cite{Osi13}.
In particular, every acylindrical action on a hyperbolic space with unbounded orbits will be partially WPD. Let us list some examples of partially WPD group actions on hyperbolic spaces. These examples can be found, among other places, in \cite[Section 2.3]{Osi18}.

\begin{enumerate}
\item[(a)] The natural action of an infinite hyperbolic group $G$ on its Cayley graph with respect to a finite generating set. 
\item[(b)]  The action of any amalgamated free product $G=A\ast_C B$ (respectively, an $HNN$-extension $G=A\ast_{C^t=D}$) on the associated Bass-Serre tree provided $A\ne C\ne B$ (respectively, $C\ne A\ne D$) and there exists $g\in G$ such that $|g^{-1}Cg \cap C|<\infty$).
This class of examples includes the action of the fundamental group of a non-geometric closed $3$-manifold on the Bass-Serre tree associated with the JSJ decomposition.

\item[(c)] The action of the mapping class group of a closed surface of genus $g\ge 2$ on the curve complex. 

\item[(d)] The action of $Out(F_n)$ on the free factor complex. 

\item[(e)] The action of a right angled Artin group on its extension graph.
\end{enumerate}

Every non-(virtually cyclic) group $G$ admitting a partially WPD action on a hyperbolic space has a maximal finite normal subgroup $E(G)$ called the \textit{finite radical} of $G$ (see \cite[Theorems~6.14 and 6.8]{DGO}). 

A subgroup $H\leqslant G$ is said to be \textit{convex cocompact} (with respect to an action $G \curvearrowright S$) if $H$ acts on $S$ properly with quasi-convex orbits. In the case when $G$ is a hyperbolic group acting properly and coboundedly on $S$, a subgroup of $G$ is convex-cocompact if and only if it is quasi-convex in $G$.
By $\Sub_{\infty}^{cc}(G\curvearrowright S)$ we denote the set of all convex cocompact subgroups of infinite index in $G$. Note that this set is non-empty if $G$ is not virtually cyclic and contains at least one loxodromic element $g$, as in this case $\langle g \rangle \in \Sub_{\infty}^{cc}(G\curvearrowright S)$.
It is easy to see that a conjugate of a convex cocompact subgroup is again convex cocompact (see Lemma~\ref{lem:cong_of_cc_is_cc}), thus $G$ acts on  $\Sub_{\infty}^{cc}(G\curvearrowright S)$ by conjugation. 

We are now ready to state our first main result. As usual, for a set $A$ in a topological space, we denote by $\overline{A}$ its closure.

\begin{thm}\label{Thm:main1}
Let $G$ be a non-virtually-cyclic group acting on a hyperbolic space $S$. Suppose that the action is partially WPD and $G$ has trivial finite radical. Then the action of $G$ on $\overline{\Sub_{\infty}^{cc}(G\curvearrowright S)}$ by conjugation is topologically $\mu$-mixing for every probability measure $\mu\in \Prob(G)$ whose support is finite, symmetric and generates $G$.
\end{thm}

This theorem is a special case of Theorem \ref{Thm:CC}, which is stated for a more general class of measures (see Definition~\ref{defn:permissible}) and for groups which are not necessarily finitely generated.

Recall that a hyperbolic group $G$ is \emph{locally quasi-convex} if every finitely generated subgroup is quasi-convex. In this case, the closure of the space of quasi-convex subgroups of $G$ coincides with $\Sub_\infty(G)$. Known examples of locally quasi-convex hyperbolic groups are free groups, surface groups, certain one-relator groups \cite{MW}, and fundamental groups of finite graphs of free groups with cyclic edge groups that do not contain any Baumslag-Solitar subgroups, e.g., amalgamated products of two free groups over a cyclic subgroup which is maximal in one of the factors \cite[Theorem~3.6]{BW}. On the other hand, Agol's virtual fibering theorem implies that closed hyperbolic $3$-manifold groups are not locally quasi-convex. Theorem \ref{Thm:main1} immediately gives the following.

\begin{cor}\label{Cor:main1}
Suppose that $G$ is a locally quasi-convex hyperbolic group that is not virtually-cyclic and has trivial finite radical. Then the action of $G$ on the space $\Sub_\infty(G)$ by conjugation is topologically $\mu$-mixing for every probability measure $\mu\in \Prob(G)$, whose support is finite, symmetric and generates $G$.
\end{cor}

For relatively hyperbolic groups, the following analogue of Theorem~\ref{Thm:main1} can be proven in much the same way. Recall that a relatively hyperbolic group $G$ is said to be \textit{elementary} if it is virtually cyclic or one of the peripheral subgroups coincides with $G$.

\begin{thm}\label{Thm:main2}
Suppose that $G$ is a non-elementary relatively hyperbolic group with trivial finite radical. Then the action of $G$ on the closure of the space of infinite index relatively quasi-convex subgroups in $\Sub(G)$ by conjugation is topologically $\mu$-mixing, for every probability measure $\mu\in \Prob(G)$ whose support is finite, symmetric and generates $G$.
\end{thm}
Note that relatively quasi-convex subgroups of a relatively hyperbolic group $G$ need not always be finitely generated, even when $G$ itself is finitely generated.

Using relative hyperbolicity of finitely generated groups with infinitely many ends, we obtain the following.

\begin{cor}\label{Cor:main2}
Let $G$ be a finitely generated group with infinitely many ends that has no non-trivial finite normal subgroups. Then the action of $G$ on the space $\Sub_\infty(G)$ is topologically $\mu$-mixing for every  measure $\mu\in \Prob(G)$, whose support is finite, symmetric and generates $G$.
\end{cor}

More generally, the statement of Corollary~\ref{Cor:main2} applies to all non-elementary finitely generated relatively hyperbolic groups that are \emph{locally quasi-convex} (i.e., every finitely generated subgroup is relatively quasiconvex). Examples of such groups include limit groups \cite{Dah}, and many fundamental groups of finite graphs of locally quasi-convex relatively hyperbolic groups with Noetherian edge groups  (e.g., fundamental groups of graphs of free groups with cyclic edge groups, where at least one vertex is non-abelian) \cite{BW}.

Two remarks are in order. First, we note that the assumption that the finite radical of $G$ is trivial cannot be omitted in both theorems. Second, the local quasi-convexity assumption cannot be omitted in Corollary \ref{Cor:main1} (see Example \ref{Ex:Rips} for details).

Combining Theorems \ref{Thm:main1}, \ref{Thm:main2} and Corollaries \ref{Cor:main1}, \ref{Cor:main2}  with Proposition \ref{Prop:HT}, we obtain that the corresponding actions are highly topologically transitive. At the final stage of our work, we learned that this result was independently obtained by Azuelos and Gaboriau \cite{AG} in certain particular cases including hyperbolic groups and groups with infinitely many ends. Their proof is different from ours and does not seem to imply the stronger topological $\mu$-mixing property. 

\subsection{Random walks}

Our proofs of Theorems \ref{Thm:main1} and \ref{Thm:main2} make use of the results obtained in \cite{AH21}, where Abbott and the first author studied the interactions between random walks and a fixed convex cocompact subgroup $H$ of a group $G$ with a partially WPD action on a hyperbolic space. To apply these results, we need to verify that $H$ satisfies an additional technical hypothesis, namely the existence of a \emph{transverse element} (see Definition \ref{def:transv_elem}). In \cite{AH21}, this additional condition was verified in many particular cases of interest. In this paper, we prove the existence of transverse elements in general settings. This is crucial for the proofs of our main theorems and also allows us to strengthen the main result of \cite{AH21} as follows (see Section~\ref{sec:rw_htt} for details).

\begin{thm}\label{thm:intro-ccrw}
Let $G$ be a non-(virtually-cyclic) group admitting a partially WPD action on a hyperbolic space $S$. Suppose that $G$ has trivial finite radical and let $H$ be a convex cocompact subgroup of $G$. Then for any probability measure $\mu\in \Prob(G)$, whose support is finite, symmetric, and generates $G$, we have  
\[
\lim\limits_{n\to \infty} \mu^{\ast n} \left(\left\{  g\in G\; \left|\; 
\begin{array}{c}
    g {\rm \; is \; loxodromic,}\\
  \langle H, g\rangle\cong H\ast \langle g\rangle, \; {\rm and}\\
  \langle H, g \rangle\; {\rm is\;  convex\; cocompact}
\end{array} \right.\right\}\right)=1.
\]
\end{thm}

Note that even the existence of an element $g\in G$ such that $\langle H, g\rangle\cong H\ast \langle g\rangle$ is a new result in this level of generality, though it was previously known in many particular cases \cite{Arz,M-3,ArzMin,AD,AC,AH21}. Also, applying Theorem \ref{thm:intro-ccrw} to right angled Artin groups and hierarchically hyperbolic groups allows us to confirm \cite[Conjecture 6.14]{AH21} and \cite[Conjecture 6.15]{AH21} respectively, see Section \ref{sec:rw_htt}. There is also an analogue of Theorem \ref{thm:intro-ccrw} for relatively quasi-convex subgroups of relatively hyperbolic groups, but this can already be found in \cite[Theorem 6.3]{AH21}. 

\subsection{Organization of the paper} 
In Section \ref{sec:back} we give some background on $\delta$--hyperbolic spaces, random walks and topological dynamics. In Sections \ref{sec:apt} and \ref{sec:trans} we establish the existence of transverse loxodromic elements for convex cocompact subgroups. The existence of transverse elements is used in Section \ref{sec:rw_htt} to strengthen the main results of \cite{AH21} and to prove Theorem \ref{Thm:main1}. Finally,  Theorem \ref{Thm:main2} and Corollary \ref{Cor:main2} are proved in Section~\ref{sec:relhyp}.

\section{Background}\label{sec:back}
\subsection{\texorpdfstring{$\delta$}{delta}-Hyperbolic geometry}\label{sec:hypgeo}
Let $(S, d)$ be a geodesic metric space. For $x, y\in S$, we will use $[x, y]$ to denote a geodesic from $x$ to $y$ in $S$. If $p$ is a path in $S$, we denote the initial point and the terminal point of $p$ by $p_-$ and $p_+$ respectively, and we denote the length of $p$ by $\|p\|$. A rectifiable path  $p$ is $(\lambda, c)$-quasi-geodesic, for some constants $\lambda \ge 1$ and $c \ge 0$, if for any
subpath $q$ of $p$ we have $\|q\| \le \lambda \d(q_-, q_+)+c$.

Given a metric space $(S,\d)$, a point $x \in S$, a subset $Y \subseteq S$ and a number  $C \ge 0$, by a \emph{$C$-neighborhood of $Y$ in $S$}, $\mathcal{N}_C(Y)$,
we will always mean the closed neighborhood of the form $\{s \in S \mid \exists\, y \in Y \mbox{ such that } \d(s,y)\le C\}$. We will also write $\d(x,Y) \le C$
to say that $x \in \mathcal{N}_C(Y)$, i.e., that $\d(x,y) \le C$ for some $y \in Y$.

A geodesic triangle $\Delta$ in $S$ is said to be \emph{$\delta$-slim}, for some $\delta \ge 0$, if every side of $\Delta$ is contained in the closed
$\delta$-neighborhood of the union of its two other sides. A geodesic metric space $(S,\d)$ is \emph{$\delta$-hyperbolic} if every geodesic triangle in $S$ is $\delta$-slim, for
some fixed $\delta \ge 0$ (see \cite[Chapter~1]{Mih} for equivalent definitions and more background on $\delta$-hyperbolicity). We will say that $S$ is \emph{hyperbolic} if it is $\delta$-hyperbolic, for some $\delta \ge 0$.

For the remainder of this subsection, $(S,\d)$ will be a $\delta$-hyperbolic geodesic metric space.
The next well-known property is sometimes called the stability of quasi-geodesics in a hyperbolic space (cf. {\cite[Theorem~1.7 in Chapter~III.H]{BH}}):

\begin{lem}\label{lem:stab_of_qgeod} For any $\lambda \ge 1$ and $c \ge 0$ there exists a constant $\varkappa \ge 0$ such that the following holds.
If $p$ and $q$ are paths in $S$ joining the same two points $p_-=q_-$ and $p_+=q_+$, where $p$ is geodesic and $q$ is $(\lambda,c)$-quasi-geodesic, then
$p$ is contained in the $\varkappa$-neighborhood of $q$ and $q$ is contained in the $\varkappa$-neighborhood of $p$ in $S$.
\end{lem}

Given three points $x,y,s \in S$, the \emph{Gromov product of $x$ and $y$ with respect to $s$} is defined by the formula
\begin{equation}\label{eq:Gr-prod-def}
(x\,|\,y)_s=\frac12 \left(\d(x,s)+\d(y,s)-\d(x,y)\right).
\end{equation}

The next lemma was proved in \cite[Lemma 21]{Olsh}. % (see \cite[Lemma 3.5]{M-2} for the last claim).
The constants involving $\delta$ are different from the constants in the original statement of  \cite[Lemma 21]{Olsh}, since we are using a
different definition of $\delta$-hyperbolicity (our definition, in terms of $\delta$-slim triangles, implies the $12\delta$-hyperbolicity of Gromov products, used in \cite{Olsh};
see \cite[Chapter~2, Propositon~2.1]{Mih}).

\begin{lem}\label{lem:piecewise_qgeod} Let $S$ be a $\delta$-hyperbolic geodesic metric space, for some $\delta \ge 0$.
Suppose that $C_0 \ge 168\delta$ and $C_1> 12(C_0+12\delta)$  are some constants.
If $x_0,x_1,\dots,x_m$ are points of $S$ satisfying $\d(x_{i-1},x_{i}) \ge C_1$ for $i=1,\dots,m$,  and $(x_{i-1}\,|\,x_{i+1})_{x_i} \le C_0$ for $i=1,\dots,m-1$, and $q$
is any geodesic segment joining $x_0$ with $x_m$, then  $\d(x_i,q) \le 2C_0$ for all $i=1,\dots,m$. %, and $\d(x_0,x_m) \ge m C_1 /2$.
\end{lem}

For a constant $\eta\geq 0$, a subset $Y \subseteq S$ is \emph{$\eta$-quasi-convex} (or simply \emph{quasi-convex}) if, whenever $p$ is a geodesic with endpoints in $Y$, $p\subseteq \mathcal{N}_\eta(Y)$. 

Suppose, now, that $G$ is a group acting by isometries on $S$. This action is (\emph{metrically}) \emph{proper} if for every $s \in S$ and each $R >0$ the subset $\{g \in G \mid \d(s,g(s)) \le R\}$ is finite. 

We will often be interested in subgroups or subsets of $G$ which have quasi-convex orbits, that is $H\subseteq G$ such that for some $s\in S$, the orbit $H(s)$ is a quasi-convex subset of $S$. Note that this is independent of the basepoint $s$. 

\begin{lem}[{\cite[Remark~2.1]{AMS}}] \label{lem:qc_orbits}
Let $H\subseteq G$ such that $H(s)$ is $\eta$-quasi-convex in $S$. Then for any $t\in S$, $H(t)$ will be $\eta'$-quasi-convex for $\eta'=2\delta+\eta+2\d(s,t) \ge 0$.
\end{lem}

The next lemma easily follows from the standard fact that the union of quasi-convex subsets in a hyperbolic metric space is again quasi-convex (see \cite[Lemma~2.1]{Min-1}).

\begin{lem}\label{lem:union_of_qc} Suppose that $H,K \leqslant G$ and $s \in S$ are such that the orbits $H(s)$ and $K(s)$ are quasi-convex. Then for $Y=(H \cup K) \setminus \{1\}$ the orbit $Y(s)$ is also quasi-convex in $S$.    
\end{lem}

\begin{lem}\label{lem:quadr-H} Suppose that $s \in S$ and $H\leqslant G$ is a subgroup such that the $H$-orbit $ H(s)$ is $\eta$-quasi-convex, for some $\eta \ge 0$.
Assume that we are given $\lambda\ge 1$ and $c \ge 0$, and let $\varkappa=\varkappa(\lambda,c)$ be the constant provided by Lemma~\ref{lem:stab_of_qgeod}. If $q$ is a
$(\lambda,c)$-quasi-geodesic path in $S$ with $\d(q_-, H(s)) \le \e$, $\d(q_+, H(s)) \le \e$, for some $\e \ge 0$, and $x\in q$ is any point then
$\d(x, H(s)) \le \varkappa+2\delta+\e+\eta$. Moreover, if, additionally,  $x$ satisfies $\d(q_-,x) >\e+\varkappa+2\delta$ and $\d(x,q_+) >\e+\varkappa+2\delta$
then $\d(x, H(s)) \le \varkappa+2\delta+\eta$.
\end{lem}

\begin{proof} Take $h_1,h_2 \in H$ so that $\d(q_-,h_1(s)) \le \e$ and
$\d(q_+,h_2(s)) \le \e$, and consider a geodesic quadrangle $Q$ in $S$ with vertices $h_1 (s)$, $q_-$, $q_+$ and $h_2(s)$.
If $a,b$ are adjacent vertices of $Q$ we will use $[a,b]$ to denote the side of $Q$ joining $a$ to $b$.
Then, according to Lemma \ref{lem:stab_of_qgeod}, there is $y \in [q_-,q_+]$ such that $\d(x,y) \le \varkappa$.
Observe that $y$ belongs to the $(2\delta+\e)$-neighborhood of $[h_1(s),h_2(s)]$ (because $Q$
is $2\delta$-slim), and the latter segment lies in the $\eta$-neighborhood of $ H(s)$. Therefore $\d(x, H(s)) \le \d(x,y)+\d(y, H(s)) \le \varkappa+2\delta+\e+\eta$.

To prove the second claim of the lemma, assume that $\d(q_-,x) >\e+\varkappa+2\delta$ and $\d(x,q_+) >\e+\varkappa+2\delta$. Then
the triangle inequality implies that
\[ \d(q_-,y) \ge \d(q_-,x)-\d(x,y) > \e+2\delta \mbox{ and } \d(y,q_+) \ge \d(x,q_+)-\d(x,y) > \e+2\delta.
\]
Therefore $\d(y,[h_1(s),q_-]) \ge \d(y,q_-)-\|[h_1(s),q_-]\|> \e+2\delta-\e=2\delta$, and, similarly, $\d(y,[q_+,h_2(s)])>2\delta$.
But since $Q$ is $2\delta$-slim, the latter shows that $y$ must belong to the $2\delta$-neighborhood of the side $[h_1(s),h_2(s)]$.
On the other hand, by $\eta$-quasi-convexity of $ H(s)$, this side is contained in the $\eta$-neighborhood of $ H(s)$.
Thus $\d(y, H(s)) \le 2\delta+\eta$, which implies that $\d(x, H(s)) \le \varkappa+2\delta+\eta$, as claimed.
\end{proof}

%%%%%%%%%%%%%%%%%%

Recall that, given $s\in S$, an element $g \in G$ is called \emph{loxodromic} if the map $\Z\to S$, $n\to g^n(s)$, is a quasi-isometric embedding. Note that, up to changing constants, this is again independent of the choice of the basepoint $s\in S$. Now we fix a geodesic $[s, g(s)]$ and form a path by concatenating the geodesics $g^n([s, g(s)])$, for all $n\in \Z$. This bi-infinite path is called \emph{a quasi-geodesic axis for $g$ based at $s$}, and it is a $g$-invariant quasi-geodesic on which $g$ acts as a non-trivial translation. 

If $g \in G$ is loxodromic, and $x,y$ are two points on some quasi-geodesic axis $\ell$, of $g$,
we will write $\ell_{[x,y]}$ to denote the segment of $\ell$ (or of $\ell^{-1}$) from $x$ to $y$.

Associated to each loxodromic element $g$ are two points $g^{+\infty}, g^{-\infty}$ in the Gromov boundary $ \partial S$, of $S$, defined by $g^{+ \infty}=\displaystyle\lim_{n\to +\infty} g^n(s)$ and $g^{- \infty}=\displaystyle\lim_{n\to-\infty} g^n(s)$, for some (equivalently, any) $s\in S$. Loxodromic elements $g,h \in G$ are called \emph{independent} if $\{g^{+\infty}, g^{-\infty}\}\cap\{h^{+\infty}, h^{-\infty}\}=\emptyset$. The action of $G$ on $S$ is called \emph{non-elementary} if $G$ contains at least two independent loxodromic elements. In this case, $G$ will necessarily contain an infinite collection of pairwise independent loxodromic elements, see \cite[Section~3]{Osi13}.

\begin{defn}\label{defn:wpd}    
A loxodromic element $g\in G$ is called a \emph{WPD} element if for all $s\in S$ and $\e\geq 0$ there exists $N\in\N$ such that
\[
|\{u\in G\mid \d(u(s),s)\leq \e, \d\left(ug^N(s),g^N(s)\right)\leq \e\}|<\infty.
\]
\end{defn}

The action of $G$ on $S$ is called \emph{WPD} if every loxodromic element is a WPD element and the action is called \emph{partially WPD} if $G$ contains at least one loxodromic WPD element.

Recall that a group is said to be \emph{elementary} if it has a cyclic subgroup of finite index.

\begin{lem}{\rm \cite[Lemma 6.5, Corollary 6.6]{DGO}}\label{lem:elemrem}
Every loxodromic WPD element $g \in G$ is contained in a unique maximal elementary subgroup
$E_G(g) \leqslant G$, which can be defined by
\begin{equation}\label{eq:def_of_elem}
E_G(g)= \{x \in G \mid \exists~m,n\in \Z\setminus \{ 0\}  \mbox{ such that }  x^{-1}g^mx=g^n\}.
\end{equation}
\end{lem}

A subgroup $H\leqslant G$ which acts properly on $S$ and has quasi-convex orbits is called \emph{$S$-convex compact}, or simply \emph{convex cocompact} when the space $S$ is understood. Equivalently, $H$ is finitely generated and, for every $s \in S$,  the orbit map $h \mapsto h(s)$ is a quasi-isometric embedding of $H$ (with the word metric from a finite generating set) into $S$ (see, for example, \cite[Lemma 2.5]{AMS}). It follows that every convex cocompact subgroup is a hyperbolic group in the sense of Gromov. Note that a cyclic subgroup $\langle g\rangle$ is $S$-convex cocompact if and only if $g$ is loxodromic with respect to the action on $S$.

\begin{lem} \label{lem:cong_of_cc_is_cc} Let $G$ be a group acting by isometries on a hyperbolic geodesic metric space $S$. If $H \leqslant G$ is convex cocompact with respect to this action then so is $gHg^{-1}$, for every $g \in G$.
\end{lem}

\begin{proof} Assume that $H$ is convex cocompact and
take any $g \in G$. For any $s \in S$ set $s'=g^{-1}(s) \in S$ and observe that 
\[\d\left(ghg^{-1}(s),s\right)=\d(h(s'),s'), \text{ for all } h \in H,\]
whence properness of the action of $H$ on $S$ implies properness of the action of $gHg^{-1}$ on $S$.

The orbit $gHg^{-1}(s)$ is the $g$-translate of the  orbit $H(s')$. Since $H(s')$ is quasi-convex in $S$, the same holds for $gHg^{-1}(s)$. Therefore $gHg^{-1}$ is also convex cocompact.
\end{proof}

\subsection{Permissible measures and random walks}\label{sec:rw}

Let $G$ be a countable group which has a non-elementary action on a hyperbolic metric space $S$ and let $\mu\colon 2^G\to\R_{\geq 0}$ be a probability measure on $G$. Let $\mathrm{Supp}(\mu)$ denote the support of $\mu$, that is $\mathrm{Supp}(\mu)=\{g \in G \mid \mu(g)>0\}$,  and let $\Gamma_\mu$ denote the subsemigroup of $G$ generated by $\mathrm{Supp}(\mu)$. We say that $\mu$ is \emph{bounded} if $\mathrm{Supp}(\mu)(s)$ is bounded, for any (equivalently, all)  $s \in S$. The measure $\mu$ is \emph{reversible} if $\Gamma_\mu$ is a subgroup of $G$, \emph{non-elementary} if $\Gamma_\mu$ contains at least two independent loxodromic elements, and \emph{WPD} if $\Gamma_\mu$ has at least one loxodromic WPD element (with respect to the induced action on $S$). When $\mu$ is reversible, non-elementary and WPD, there is a unique maximal finite subgroup of $G$ normalized by $\Gamma_\mu$ which we denote by $E(\mu)$ (this follows from a combination of \cite[Theorem~1.4]{Osi13} with \cite[Lemma~5.5]{H16} or with \cite[Lemma~5.6]{AMS}). Finally, we say that $\mu$ is \emph{symmetric} if $\mu(g^{-1})=\mu(g)$, for all $g \in G$, and $\mu$ has \emph{generating support} if $\Gamma_\mu=G$.

\begin{defn}\cite[Definition 2.8]{AH21}\label{defn:permissible}
Let $G$ be a group with a non-elementary, partially WPD action on a hyperbolic metric space $S$. Then a probability measure $\mu$ on $G$ is called \emph{permissible} if $\mu$ is bounded, reversible, non-elementary, WPD and $E(\mu)=E(G)$.
\end{defn}

Note that when $G$ is finitely generated, any probability measure whose support is a finite, symmetric generating set of $G$ will be permissible. For example, the uniform measure on a finite symmetric generating set will be a permissible measure on $G$ in this case. In general, whenever $G$ has a non-elementary partially WPD action on a hyperbolic metric space $S$,  $G$ has a finitely generated subgroup $F$ such that $F$ contains loxodromic WPD elements, the action of $F$ on $S$ is non-elementary, and the maximal finite subgroup of $G$ normalized by $F$ is equal to $E(G)$ (this follows from \cite[Theorem~1.4]{Osi13} and \cite[Lemma~6.17]{DGO}). In this case any probability measure whose support is a finite symmetric generating set of $F$, for example the uniform measure on this set, will be permissible. In particular, permissible probability measures exist for any such action. Moreover, if $A$ is any finite symmetric subset of $G$ such that $F\leqslant \langle A\rangle$, then any measure $\mu$ with $\mathrm{Supp}(\mu)=A$ is permissible. Thus any finite subset of $G$ is contained in the support of some permissible measure.

Let $\mu$ be a probability measure on a group $G$. A \emph{random walk} of length $n$ with respect to $\mu$ means a random element $w_n=g_1 \dots g_n$, where each $g_i$ is chosen independently according to the measure $\mu$. In other words, $w_n$ is a random element of $G$ chosen with respect to the probability measure $\mu^{\ast n}$, the $n$-fold convolution of $\mu$ with itself. We say that $w_n$ satisfies some property $P$ with probability $\alpha$ if $\mu^{\ast n}(\{g\in G\;|\; g \text{ satisfies } P\})=\alpha$. If $(P_n)$ is a sequence of such properties, we say that $w_n$ satisfies $P_n$ \emph{asymptotically almost surely} if the probability  that $w_n$ satisfies $P_n$ tends to $1$ as $n$ goes to infinity, i.e.
\[
\lim_{n\to\infty}\mu^{\ast n}(\{g\in G\;|\; g \text{ satisfies } P_n\})=1.
\]

For example, if $G$ is a countable group with a non-elementary action on a hyperbolic metric space and $\mu$ is a non-elementary probability measure on $G$, then Maher-Tiozzo showed that a random walk $w_n$ with respect to $\mu$ will be loxodromic asymptotically almost surely \cite[Theorem 1.4]{MT18}. Moreover, they also show the following.

\begin{thm}[{\cite[Theorem 1.2]{MT18}}]\label{thm:drift}
Let $G$ be a group with a non-elementary, partially WPD action on a hyperbolic metric space $S$, let $\mu$ be a  permissible probability measure on $G$, and let $w_n$ be a random walk with respect to $\mu$. Then for all $s\in S$, there exists a constant $D>0$ such that for all $\e>0$, 
\[
(1-\e)Dn\leq d(s, w_n(s))\leq (1+\e)Dn
\]
asymptotically almost surely.
\end{thm}

With this terminology, we can also rephrase the definition of a topologically $\mu$-mixing action (see Definition \ref{def:mu-mixing}) in terms of random walks. That is, the action of $G$ on a topological space $X$ is topologically $\mu$-mixing if and only if for any non-empty open subsets $U$ and $V$ of $X$ we have $w_n(U)\cap V\neq\emptyset$ asymptotically almost surely.

Now, suppose we have a sequence $(\mu_i)$ of permissible probability measures on $G$. For a fixed $n, k\in \N$, we consider the product measure $\mu_1^{\ast n}\times...\times \mu_k^{\ast n}$ on $G^k$ and let $\mu_{n, k}$ denote the push-forward of this product measure to $\Sub(G)$ under the map $G^k\to\Sub(G)$ which sends $(g_1,...,g_k)\to\langle g_1,...,g_k\rangle$. A random element with respect to $\mu_{n, k}$ can be thought of as a subgroup of the form  $\langle w_{n, 1},\dots, w_{n, k}\rangle$ where each $w_{n, i}$ is a random walk of length $n$ with respect to $\mu_i$. We will say that this subgroup $\langle w_{n, 1},\dots, w_{n, k}\rangle$ satisfies some property $P$ with probability $\alpha$ if $\mu_{n, k}(\{H\in \Sub(G)\;|\; H \text{ satisfies } P\}=\alpha$. For a sequence of properties $P_n$ of subgroups, we say that a random $k$-generated subgroup with respect to the sequence $(\mu_i)$ satisfies $P_n$ asymptotically almost surely if

\[
\lim_{n\to\infty}\mu_{n, k}(\{H\in \Sub(G)\;|\; H \text{ satisfies } P_n\})=1.
\] 
 
 For example, if $G$ is a countable group with a non-elementary action on a hyperbolic metric space and $\mu$ is a non-elementary measure on $G$, then random $k$-generated subgroups with respect to the constant sequence $(\mu)$ are convex cocompact and isomorphic to free groups of rank $k$ asymptotically almost surely \cite[Theorem 1.2]{TayTio}.

\section{Topological dynamics}\label{subsec:top_dyn}

\subsection{Topological transitivity and \texorpdfstring{$\mu$}{mu}-mixing}
Recall that a topological space $X$ is called \textit{Polish} if $X$ is separable and completely metrizable.

Suppose that $G$ acts by homeomorphisms on a topological space $X$. This action  is said to be \emph{topologically $k$-transitive}, for some $k\in \N$, if for any non-empty open sets $U_1,...,U_k, V_1,...,V_k$ in $X$, there exist $g\in G$ such that $g(U_i)\cap V_i\neq \emptyset$, for all $1\leq i\leq k$. Thus, topological $1$-transitivity is the same as topological transitivity, and an action is highly topologically transitive if and only if it is topologically $k$-transitive for all $k\in \N$.

We refer to \cite{CKN} for a comparison of these transitivity conditions and other notions from topological dynamics in various settings. In particular, the following standard consequence of the Baire Category Theorem can be found there. 

\begin{lem}[{\cite[Proposition~1]{CKN}}]\label{lem:tt_dense_orbit}
Suppose $G$ has a topologically transitive action on a Polish space $X$. Then there exists $x\in X$ such that the orbit $G(x)$ is dense in $X$.
\end{lem}

We begin by discussing an elementary example showing that topological $\mu$-mixing defined in the introduction is strictly stronger than topological transitivity and weaker than topological mixing.

\begin{ex}\label{Ex:mix}
Consider the action of $G=\ZZ$ on $\mathbb S^1=\mathbb R/\ZZ$ by irrational rotation. It is well-known and easy to prove that this action is topologically transitive. On the other hand, let $U$, $V$, $W$ be open arcs in $\mathbb S^1$. If these arcs are small and far from each other, the sets $N(U,V)$ and $N(U,W)$ will be disjoint as the action of $G$ on $\mathbb S^1$ is isometric. It follows that these sets can not simultaneously have ``hitting probability $1$" with respect to any random walk on $G$; that is, the action of $G$ on $\mathbb S^1$ is not topologically $\mu$-mixing with respect to any $\mu\in \Prob(G)$.
\end{ex}

Let us now relate topological $\mu$-mixing to high topological transitivity.  

\begin{proof}[Proof of Proposition \ref{Prop:HT}]
Let $G$ be a group $G$, $\mu \in \Prob(G)$. Suppose that $G$ acts on a topological space $X$ and the action is topologically $\mu$-mixing. For any $k\in \NN$ and any non-empty open sets $U_1,...,U_k, V_1,...,V_k\subseteq X$, we have
$$
\begin{array}{rcl}
\lim\limits_{n\to\infty} \mu^{\ast n} \left(\bigcap_{i=1}^k N(U_i, V_i)\right) & = & \lim\limits_{n\to\infty} \mu^{\ast n} \left(G\setminus \bigcup_{i=1}^k (G\setminus N(U_i, V_i))\right)  \ge\\ &&\\&& 1- \lim\limits_{n\to\infty} \sum_{i=1}^k  \mu^{\ast n}  (G\setminus N(U_i, V_i)) = \\ &&\\&& 1 - \sum_{i=1}^k \Big(1- \lim\limits_{n\to\infty} \mu^{\ast n} (N(U_i, V_i))\Big)= 1. 
\end{array}
$$
In particular, $\bigcap_{i=1}^k N(U_i, V_i)\ne \emptyset$ and the proposition follows.
\end{proof}

It is easy to show that, in general, high topological transitivity does not imply $\mu$-mixing. 

\begin{ex}
Let $G$ be a group acting highly topologically transitively on a space $X$ and let $H\leqslant G$ be a subgroup such that the induced action of $H$ on $X$ is not highly topologically transitive (e.g., we can take $H=\{ 1\}$). For any $\mu \in \mathrm{Prob}(G)$ supported on $H$, the action of $G$ is not topologically $\mu$-mixing by Proposition~\ref{Prop:HT}.
\end{ex}

However, in the above example the support of the measure $\mu $ does not generate $G$. Our next goal is to demonstrate that topological $\mu$-mixing is strictly stronger than high topological transitivity even in the case when the measure has generating support. The construction below is inspired by the action of Thompson's group $V$ on the Cantor set $\{0,1\}^{\N}$, and, more specifically, by the discussion of this action in the work  of Bleak and Quick \cite{BQ}.

\begin{prop}\label{prop:high_top_act_not_mixing} There exist a group $G$ generated by a finite symmetric set $A$, a compact Polish space $X$, and an action of $G$ on $X$ by homeomorphisms that is highly topologically transitive but not $\mu$-mixing for the uniform measure $\mu$ supported on $A$.
\end{prop}

\begin{proof} Denote by $F_3=F(x,y,z)$ the free group on the set $\{x,y,z\}$ and let $F_2=F(x,y)$ be its subgroup generated by $\{x,y\}$. Take $G=F_2*S_{18}$, where $S_{18}$ is the symmetric group on $18$ letters. 

Let $T$ be the Cayley graph of $F_3$ with respect to $\{x,y,z\}$. We think of $T$ as a rooted tree, where the vertices correspond to reduced words over $\{x,y,z\}^{\pm 1}$, with the root at the empty word (the identity element of $F_3$). We let $X$ be the visual (Gromov) boundary of $T$, which is simply the set of all infinite reduced words over $\{x,y,z\}^{\pm 1}$. The topology on $X$ is the standard one: two such infinite words are close if they share a long common prefix. Thus, a basic open set in $X$ consists of all infinite words that start with a given reduced word $w \in F_3$; we will call this open set the \emph{cone of $w$}, denoted $\cone(w)$.
Clearly, $X$ is a perfect closed subset of the Cantor space $(\{x,y,z\}^{\pm 1})^{\N}$, and, hence, $X$ is itself (homeomorphic to) a Cantor space. For every non-empty reduced word $w \in F_3$ and every $n \in \N\cup\{0\}$ we denote by $\cone_n(w)$ the set of all finite reduced words $u \in F_3$ of length $\|u\|=\|w\|+n$ starting with $w$.

The linear order $x<x^{-1}<y<y^{-1}<z<z^{-1}$ on $\{x,y,z\}^{\pm 1}$ induces a lexicographic (total) order on any collection of reduced words in $F_3$ of length $m$, for each $m \in \N$. It also gives rise to a total order on $X$. Given any two non-empty reduced words $u,v \in F_3$ and arbitrary $n \in \N\cup\{0\}$, there is a unique order-preserving bijection from $\cone_n(u)$ to $\cone_n(v)$; when $n \to \infty$, these bijections induce an order-preserving bijection $$\xi_{u,v}\colon \cone(u) \to \cone(v).$$

Let $\Sigma \subseteq F_3$ be the set of all reduced words of length $2$ in $F_3$, and let $\Omega=\Sigma\setminus F_2$, so that $|\Sigma|=30$ and $|\Omega|=18$. We treat $S_{18} \cong \mathrm{Sym}(\Omega)$ as the subgroup of the full symmetric group $\mathrm{Sym}(\Sigma)$ that permutes elements of $\Omega$ and fixes pointwise all words in $\Sigma \setminus \Omega=\Sigma \cap F_2$. Thus  $S_{18}$ naturally acts on $\Sigma$ by permutations (on the left). This action induces an action of $S_{18}$ on $X$ as follows: for every $u\in \Sigma$, every $w\in \cone(u)$ and every $\sigma \in S_{18}$, we let 
$$
\sigma(w)=\xi_{u,\sigma(u)}(w).
$$
Since $\displaystyle X=\bigsqcup_{u \in \Sigma} \cone(u)$, this gives an action of $S_{18}$ on $X$ by homeomorphisms, which permutes the cones of words of length $2$ in $F_3 \setminus F_2$ and fixes pointwise all infinite reduced words whose prefixes of length $2$ belong to $F_2$.
Further, the group $F_2$ acts naturally on $X$ by left multiplication. Therefore, we can combine these two actions to obtain a left action of the free product $G=F_2*S_{18}$ on $X$ by homeomorphisms.

Given two non-empty reduced words $u,v \in F_3$ and $g \in G$ we will say that \emph{$g$ sends $\cone(u)$ to $\cone(v)$}, writing $g[\cone(u)]=\cone(v)$, if $g$ induces an order-preserving bijection between $\cone(u)$ and $\cone(v)$; in other words, we have
\[g\vert_{\cone(u)}=\xi_{u,v}.\]
%In particular, $g[\cone(u)]=\cone(u)$ means that $g$ fixes every point of $\cone(u)$. 
%Observe that if $u$ is a prefix of $u'$ and $g[\cone(u)]=\cone(v)$, for some $g \in G$, then $g[\cone(u')]=\cone(v')$, for some word $v' \in F_3$, with prefix $v$. 
The definitions immediately imply that 
\[\sigma[\cone(w)]=\cone(\sigma(w)), \text{ for all } \sigma \in S_{18} \text{ and all } w \in \Sigma, \text{ and}\]
\[a[\cone(u)]=\cone(au),~\text{ for every } a \in \{x,y\}^{\pm 1} \text{ and all  } u \in F_3, \text{ with } \|u\|\ge 2.\]

Given an element $g\in G$ and a subset $A\subseteq X$, we denote by $g\vert _A$ the restriction of $g$ to $A$.

\begin{claim}\label{claim:1}
For every integer $n \ge 2$, and any reduced word $u \in F_3 \setminus (F_2 \cup \{z^{-n}\})$, of length $n$, there is an element $f \in G$ such that 
\begin{equation}\label{eq:prop_of_f}
f[\cone(u)]=\cone(z^2)   \text{ and } f[\cone(z^{-n})]=\cone(z^{-2}).
\end{equation}
%More precisely, we have $f\vert_{\cone(u)}=\xi_{u,z^2}$ and $f\vert_{\cone(z^{-n})}=\xi_{z^{-n}, z^{-2}}$. 
\end{claim}

\begin{proof}[Proof of Claim~\ref{claim:1}]
If $n=2$ then an element of $S_{18}$ does the required job. Thus, we can assume that $n>2$. 

Let $u_2$ be the prefix of $u$ of length $2$. We first choose elements $a \in  \{x,y\}^{\pm 1}$ and $\sigma \in S_{18}$ as follows. 
\begin{itemize}
    \item[1)] If $u_2 \in \Sigma\setminus\Omega$ starts with a letter $a \in \{x,y\}^{\pm 1}$, then let $\sigma \in S_{18}$ be a permutation fixing $\cone(u_2)$ and sending $\cone(z^{-2})$ to $\cone(a z^{-1})$. 
    \item[2)] If $u_2=z^{-2}$, then take $a=x$ and let $\sigma \in S_{18}$ be any permutation sending $\cone(z^{-2})$ to $\cone(a z^{-1})$.
    \item[3)] Otherwise (if $u_2 \in\Omega$ and $u_2 \neq z^{-2}$), set $a=x$ and let $\sigma \in S_{18}$ be any permutation sending $\cone(u_2)$ to $\cone(az)$ and $\cone(z^{-2})$ to $\cone(az^{-1})$.
\end{itemize}

Now, we have arranged $a \in  \{x,y\}^{\pm 1}$ and $\sigma \in S_{18}$ in such a way that $a^{-1}\sigma [\cone(u)]=\cone (u')$, where $u' \in F_3 \setminus F_2$ has length $n-1 \ge 2$, and $a^{-1}\sigma [\cone(z^{-n})]=\cone(z^{-n+1})$ (note that in the last equality we used the fact that  $\sigma[\cone(z^{-n})]=\cone(a z^{-n+1})$, which holds because $\sigma$ induces the unique order-preserving bijection between $\cone_{n-2}(z^{-2})$ and $\cone_{n-2}(az^{-1})$). By induction on $n$, there is $h \in G$ such that $h[\cone(u')]=\cone(z^2)$ and $h[\cone(z^{-n+1})]=\cone(z^{-2})$. Hence the element $f=h a^{-1} \sigma \in G$ satisfies \eqref{eq:prop_of_f}.
\end{proof}

\begin{claim}\label{claim:2}
If $u$ is a reduced word in $F_3 \setminus F_2$ of length $n \ge 2$, then there is an element $g \in G$ such that $g$ interchanges $\cone(u)$ with $\cone(z^{-n})$ and fixes all points in $\cone(w)$, for every other reduced word $w \in F_3$ of length $n$.
\end{claim}

\begin{proof}[Proof of Claim~\ref{claim:2}]
If $u=z^{-n}$, we can take $g=1$, so assume that $u \neq z^{-n}$. Choose the element $f \in G$ according to Claim~\ref{claim:1}, and let $\tau \in S_{18}$ be the transposition interchanging $\cone(z^2)$ with $\cone(z^{-2})$. Then the element $g=f^{-1} \tau f \in G$ will interchange $\cone(u)$ with $\cone(z^{-n})$. It remains to check that $g$ fixes all points in $\cone(w)$, for a reduced word $w$ of length $n$, such that $w \neq u$ and $w \neq z^{-n}$. This is true because $G$ acts by bijections on $X$ and $\tau$ only moves points in $\cone(z^2) \cup \cone(z^{-2})$. Indeed, if $x \in \cone(w)$, then $x \in X \setminus (\cone(u) \cup \cone(z^{-n}))$, so $f(x) \in X \setminus (\cone(z^2) \cup \cone(z^{-2}))$. It follows that $\tau(f(x))=f(x)$, hence $g(x)=x$.
\end{proof}

\begin{claim}\label{claim:3}
The action of $G$ on $X$ is highly topologically transitive.
\end{claim}

\begin{proof}[Proof of Claim~\ref{claim:3}]
Suppose that $k \in \N$ and $U_1,\dots,U_k,V_1,\dots,V_k$ are non-empty open subsets in $X$. Choose $n \ge 2$ so that the following is satisfied:
there exist reduced words $u_i,v_i \in F_3\setminus F_2$ of length $n$
such that 
\begin{itemize}
    \item the words $u_{1},\dots,u_{k}$ are pairwise distinct and the words $v_{1},\dots,v_{k}$ are pairwise distinct;
    \item $\cone(u_i) \subseteq U_i$ and $\cone(v_i) \subseteq V_i$, for all $i=1,\dots,k$.
\end{itemize}
Such a choice is possible because every non-empty open set in $X$ contains a basic open subset of the form $\cone(w)$, for some $w \in F_3\setminus F_2$, and $\cone(w)$ contains $\cone(u)$, for every word $u$ such that $w$ is a prefix of $u$.

Recall that for every $m \ge 2$ the symmetric group $S_m$, on the set $\{1,\dots,m\}$, is generated by all transpositions of the form $(j,m)$, where $j=1,\dots,m-1$. Combining this fact with Claim~\ref{claim:2}, we can conclude that there exists $g \in G$ such that $g(\cone(u_i))=\cone(v_i)$, for every $i=1\dots,k$. It follows that $v_i \in g(U_i) \cap V_i \neq \emptyset$, for all $i=1,\dots,k$, as required.
\end{proof}

Let $A = S_{18}\cup \{x,y\}^{\pm 1} \subseteq G$ and let $\mu$ be the uniform measure supported on $A$. That is, we have $\mu(a)=1/|A|=1/(18!+4) $, for each $a\in A$. 

\begin{claim}\label{claim:4}
The action of $G$ on $X$  is not topologically $\mu$-mixing.
\end{claim}

\begin{proof}[Proof of Claim~\ref{claim:4}]
Consider the probability measure $\nu$ on $F_2=F(x,y)$, given by $\nu(a)=1/|A|$, for each $a \in \{x,y\}^{\pm 1}$, and $\nu(1)=1-4/|A|$. Denote by  $\phi:G \to F_2$  the natural retraction, whose kernel is the normal closure of $S_{18}$ in $G$. 

Let $(w_n)$ be a random walk on $G$ starting at $1$ and corresponding  to $\mu$. 
Then $t_n=\phi(w_n)$ is a random walk on $F_2$ starting at $1$ and corresponding to $\nu$.  
Observe that, as long as $t_m^{-1} \neq x^{-1}$, for all $m=1,\dots,n$, we have 
\[w_n^{-1}(\cone(x^2))=t_n^{-1}(\cone(x^2)) \subseteq \cone(x) \cup \cone(x^{-1}) \cup \cone(y) \cup \cone(y^{-1}).\] Indeed, this is true because, by construction, every $\sigma \in S_{18}$ fixes all points in $\cone(w)$, for every word $w \in F_2$ of length at least $2$. Therefore, the probability $q_n$ that $w_n(\cone(z)) \cap \cone(x^2) \neq \emptyset$ (which can be re-written as $w_n^{-1}(\cone(x^2)) \cap \cone(z) \neq \emptyset$) does not exceed the probability that $t_m$ hits $x$, for some $m \in \{1,\dots,n\}$. The latter probability is bounded above by the probability $r$, of the event $\{\exists\, n \in \N\cup\{0\} \mid t_n=x\}$, that the random walk $(t_n)$ hits $x$ eventually. Note that the random walk $(t_n)$ is transient on $F_2$ by \cite[Theorem~1.16]{Woess}, as the function $f\colon F_2 \to \R$, defined by $f(v)=3^{-\|v\|}$ (where $\|v\|$ denotes the word length of $v$ in $F_2$), is positive, $\nu$-superharmonic and non-constant (see \cite[Section~I.B]{Woess}). Therefore $r<1$, and since $q_n \le r$, for all $n \in \N$, we see that $q_n$ cannot tend to $1$ as $n \to \infty$. This shows that the action of $G$ on $X$ is not topologically $\mu$-mixing.
\end{proof}

Claims~\ref{claim:3} and \ref{claim:4} complete the proof of the proposition.    
\end{proof}

%\begin{rem} Since $S_{18}$ is generated by two elements, we can also take $G$ in Proposition~\ref{prop:high_top_act_not_mixing} to be the free group of rank $4$ (because it maps onto $F_2*S_{18}$).
%\end{rem}

One may observe that the action of the group $G$ on the space $X$ in Proposition~\ref{prop:high_top_act_not_mixing} is not faithful.  Indeed, let $\sigma \in S_{18}$ be the product of $5$ independent transpositions, induced by the order-preserving bijection between $\cone(z)$ and $\cone(z^{-1})$, and let $\tau \in S_{18}$ be the transposition interchanging $\cone (xz)$ with $\cone(xz^{-1})$. Then $x^{-1} \tau x$ represents the same element of $\mathrm{Homeo}(X)$ as $\sigma$, but, of course, $x^{-1} \tau x \neq \sigma$ in $G=F_2*S_{18}$. 

The image $\overline{G}$, of $G$ in  $\mathrm{Homeo}(X)$, 
may be of independent interest, so let us say a few words about it. Let $\overline{F}_2$ and $\overline{S}_{18}$ denote the images of  $F_2$ and $S_{18}$ in $\mathrm{Homeo}(X)$, respectively. Since $F_2$ and $S_{18}$ act on $X$ faithfully, we have $\overline{F}_2 \cong F_2$ and $\overline{S}_{18} \cong S_{18}$. Let $Y \subseteq X$ be the limit set of $F_2$ in $X$, in other words, $Y$ consists of all infinite reduced words over $\{x,y\}^{\pm 1}$. 
Note that  $Y$ is invariant under the actions of $F_2$ and $S_{18}$, by construction, so it is invariant under the action of $G$ on $X$. Since $S_{18}$ acts trivially on $Y$, by definition, the normal closure $H$, of $\overline{S}_{18}$ in $\overline{G}$, also acts trivially on $Y$. On the other hand,  $F_2$ acts on $Y$ faithfully, hence we can conclude that $\overline{F_2} \cap H={1}$ in $\overline{G}$. Thus $\overline{G}$ is a semidirect product $H \rtimes \overline{F_2}$, in particular it is not simple and $\overline{G}/H \cong F_2$.

It is not difficult to check that the action of $\overline{G}$ on $X$ satisfies all the conditions from a criterion of Caprace, Reid and Willis \cite[Proposition~I]{CRW} which implies that $\overline{G}$ is monolithic, i.e., there is a non-trivial normal subgroup $M \lhd \overline{G}$ which is contained in every non-trivial normal subgroup of $\overline{G}$. It follows that $\overline{G}$ cannot be acylindrically hyperbolic (indeed, by \cite[Theorem~8.1]{DGO}, every acylindrically hyperbolic group $P$ contains a non-trivial free normal subgroup, so the monolith of $P$ must be free, but every non-trivial free group has a proper non-trivial characteristic subgroup, whence $P$ cannot be monolithic). In fact, using a more recent criterion of Garrido and Reid \cite[Corollary~1.7]{GR}, the authors have verified that the monolith of $\overline{G}$ is $[H,H]$, the derived subgroup of $H$, and $[H,H]$ is an infinite simple group. We record these observations in the following statement.

\begin{prop} Working under the notation from the proof of Proposition~\ref{prop:high_top_act_not_mixing}, denote by 
$\psi:G \to \mathrm{Homeo}(X)$ the homomorphism arising from the action of $G$ on $X$, and let $\overline{G}=\psi(G)$, $\overline{F}_2=\psi(F_2)$, $\overline{S}_{18}=\psi(S_{18})$ and $H=\ll \overline{S}_{18} \rr^{\overline{G}} \lhd \overline{G}$. Then 
$\overline{F_2} \cong F_2$, $\overline{S}_{18} \cong S_{18}$ and $\overline{G} \cong H \rtimes \overline{F}_2$. Moreover, $[H,H]$ is an infinite simple group and it is the monolith of $\overline{G}$. In particular, $\overline{G}$ is not acylindrically hyperbolic.
\end{prop}

It is, therefore, reasonable to ask the following. 

\begin{ques}
 Does there exist a finitely generated acylindrically hyperbolic group $G$, a highly topologically transitive faithful action of $G$ on a compact Polish space $X$, and a symmetric measure $\mu\in \Prob(G)$ with generating support, such that the action of $G$ on $X$ is not topologically $\mu$-mixing? 
\end{ques}

\subsection{Topologically transitive actions on subgroups}
We now turn to the case where $X=\Sub_\infty(G)$. Recall that all groups under consideration are supposed to be countable and discrete.

\begin{lem}\label{lem:fg_noraml}
Suppose that $N$ is a finitely generated non-trivial normal subgroup of infinite index in $G$. Then the action of $G$ on $\Sub_\infty(G)$ is not topologically transitive. 
\end{lem}

\begin{proof}
Since $N$ is finitely generated, there is an open set $U$ in $\Sub_\infty(G)$ consisting of all infinite index subgroups that contain $N$. By normality of $N$, for all $g\in G$ every element of $g(U)$ will also contain $N$ as a subgroup. Let $V$ be the open set of all infinite index subgroups that do not contain a fixed non-trivial element of $N$. Then no element of $V$ can contain $N$ as a subgroup, hence $g(U)\cap V=\emptyset$, for all $g\in G$. Note that $U$ and $V$ are non-empty since $N\in U$ and $\{1\}\in V$.  
\end{proof}

\begin{ex}\label{Ex:Rips}
The action of an infinite hyperbolic group $G$ on $\Sub_{\infty}(G)$ may not be topologically transitive for two reasons. First, $G$ may contain a non-trivial finite normal subgroup. Second, even if we require $G$ to be torsion-free, infinite index non-trivial normal subgroups of $G$ can be finitely generated. Indeed, for every finitely presented group $Q$, there exists a short exact sequence $\{1\}\to N\to G\to Q\to \{1\}$, where $G$ is torsion-free hyperbolic and $N$ is finitely generated, by a well-known construction of Rips \cite{Rip}. 
\end{ex}

Recall that a group $G$ is \emph{invariably generated} if there exists a subset $S\subseteq G$ such that for every function $f\colon S\to G$, $G$ is generated by the subset $\{f(s)^{-1}sf(s) \mid s\in S\}$. That is, $S$ is a generating set of $G$ which will still generate $G$ if each element of $S$ is replaced by a conjugate.

\begin{prop}\label{prop:invgen}
Let $G$ be an infinite invariably generated group. Then the action $G\curvearrowright \Sub_\infty (G)$ is topologically transitive if and only if $G\cong \ZZ$.     
\end{prop}

\begin{proof} The sufficiency is obvious, so assume that $G$ acts topologically transitively on $\Sub_\infty (G)$. First, suppose that $G$ is virtually cyclic. Then there is a finite normal subgroup $N \lhd G$ such that $G/N$ is either infinite cyclic or infinite dihedral (see \cite[Lemma~2.5]{FarJon}). By Lemma~\ref{lem:fg_noraml}, $N$ must be trivial, so either $G \cong \Z$ or $G \cong D_\infty$ is infinite dihedral. To rule out the latter case, note that every non-trivial infinite index subgroup of $D_\infty$ is cyclic of order $2$, hence any such subgroup is an isolated point of $\Sub_\infty (D_\infty)$. Now let $a,b \in D_\infty$ be the involutions generating this group (so that $D_\infty \cong \langle a\rangle*\langle b \rangle$). Then $\{\langle a \rangle\}$, $\{\langle b \rangle\}$ are open in $\Sub_\infty (D_\infty)$ but not conjugate in $D_\infty$, therefore the action $D_\infty\curvearrowright \Sub_\infty (D_\infty)$ is not topologically transitive.

Suppose now that $G$ is not virtually cyclic and the action of $G$ on $\Sub_\infty(G)$ is topologically transitive. Lemma \ref{lem:tt_dense_orbit} implies that there exists $H\in\Sub_\infty(G)$ whose orbit is dense in $\Sub_\infty(G)$. For any $x\in G$, $\langle x\rangle\in \Sub_\infty(G)$ since $G$ is not virtually cyclic. Let $U$ be an open subset of $\Sub_\infty(G)$ consisting of infinite index subgroups containing $x$. By topological transitivity, there exists $g\in G$ such that $g^{-1}Hg\in U$, and so $H$ contains a conjugate of $x$ (namely $gxg^{-1}$). 

Now, if $S$ is any subset $G$ then every element of $S$ has a conjugate which belongs to $H$. In particular, if we replace each element of $S$ by such a conjugate we will get a subset of $H$, hence this subset cannot generate all of $G$. Therefore $G$ is not invariably generated.
\end{proof}

Since virtually solvable groups are invariably generated by \cite{Wie}, we obtain the following. 

\begin{cor}\label{cor:solv}
Let $G$ be a virtually solvable group. Then the action $G\curvearrowright \Sub_\infty (G)$ is topologically transitive if and only if $G\cong \ZZ$.
\end{cor}

\section{Apt elements}\label{sec:apt}
Throughout this section we will assume that $G$ is a group acting by isometries on a $\delta$-hyperbolic geodesic metric space $(S,\d)$, for some $\delta \ge 0$. Our first goal is to show the existence of loxodromic elements $f\in G$ which are \emph{transverse} to a given subgroup $H\leqslant G$. Loosely speaking, this means that any quasi-geodesic axis of $f$ spends a bounded amount of time inside any neighborhood of the orbit of any coset of $H$ (see Definition~\ref{def:transv_elem}). The existence of these elements will allow us to apply random walk techniques developed in \cite{AH21}. We start by studying an auxiliary notion of $H$-\emph{apt} elements which will be used to construct transverse elements in the next section.

\begin{defn}\label{defn:rwpd} Let $H \leqslant G$ be a subgroup and let $g \in G$ be an element.
We will say that $g$ is $H$-\emph{\apt} if for each $s \in S$ and every $C \ge 0$ there exist $N=N(s, C) \in \N$ and a finite subset $U=U(s, C) \subseteq G$ such that
\begin{equation}\label{eq:wpds-1}\{u \in G \mid\d(u (s),s)\le C, \d(ug^N (s),H(s))\le C \}\subseteq H U.
\end{equation}
\end{defn}

\begin{rem}
It is easy to see that \eqref{eq:wpds-1} is  equivalent to
\begin{equation}\label{eq:wpds-1.5}\{u \in G \mid\d(u (s), H(s))\le C, \d(ug^N (s),H(s))\le C \}\subseteq H U.
\end{equation}
It follows that a loxodromic element $g \in G$ is $H$-apt if and only if for every $s \in S$ and each $C \ge 0$ there are only finitely many right $H$-cosets of elements $u \in G$ that move a sufficiently long segment of a quasi-geodesic axis of $g$,  based at $s$, at most $C$-away from the orbit $H(s)$.
\end{rem}

\begin{rem}\label{rem:apt-indep_of_basepoint} The universal quantifier on $s \in S$ in Definition~\ref{defn:rwpd} can be replaced with the existence quantifier. More precisely, if $H \leqslant G$ and $g \in G$ then $g$ is $H$-\apt{} if and only if there exists $s \in S$ such that for every $C \ge 0$ there are $N=N(C) \in \N$ and a finite subset $U=U(C) \subseteq G$ such that \eqref{eq:wpds-1} holds.
\end{rem}

Indeed, this follows from the triangle inequality, implying that if $s' \in S$ then for all $a,b \in G$ we have $\d(a(s),b(s)) \le \d(a(s'),b(s'))+2\d(s,s')$.

If $H$ is a finite subgroup of $G$, then every loxodromic element of $G$ will be $H$-\apt{}, because in this case the left-hand side of \eqref{eq:wpds-1} will be empty for large enough $N$. When $H$ is convex cocompact, we will show that every loxodromic WPD element of $G$ is $H$-\apt{} (see \cref{prop:rwpds}). Also if $G$ is hyperbolic relative to a finite family of peripheral subgroups, $S$ is the corresponding relative Cayley graph and $H \leqslant G$ is relatively quasi-convex, then we will show that $g$ is $H$-\apt{} for every loxodromic $g \in G$ (see \cref{prop:rel_qc-rwpd}).

Next, we record some basic properties and useful characterizations of $H$-\apt{} elements.

\begin{lem}\label{lem:apt-inverse} Let $H \leqslant G$ and $g \in G$ be  an $H$-\apt{} element. Then $g^{-1}$ is also $H$-\apt{}.
\end{lem}

\begin{proof} Indeed, given any $s \in S$ and $C\ge 0$,  let
$N=N(s,C) \in \N$ and $U=U(s,C) \subseteq G$ be provided
by Definition~\ref{defn:rwpd}.

Define $U'=Ug^N \subseteq G$, so that $|U'|=|U|<\infty$, and suppose that $u \in G$ satisfies $\d(s,u(s)) \le C$ and $\d(ug^{-N}(s),h(s)) \le C$, for some $h \in H$.
Then, by letting $v=h^{-1}ug^{-N} \in G$, we obtain
\begin{align*}
 &\d(v(s),s) =\d(h^{-1}ug^{-N}(s),s)=\d(ug^{-N}(s),h(s)) \le C, \mbox{ and } \\   
 &\d(vg^N(s), H(s)) \le \d(h^{-1}u(s),h^{-1}(s))=\d(u(s),s) \le C.
\end{align*} 

It follows that $v \in HU$, and thus $u=hvg^N \in HU'$.
Therefore the element $g^{-1}$ is also $H$-\apt{}, as required.
\end{proof}

\begin{lem}\label{lem:equiv_rwpds}
Suppose that the induced action of $H \leqslant G$ on $S$  has quasi-convex orbits and $g \in G$ is a loxodromic element.
Then  $g$ is $H$-\apt{}  if and only if for all $s \in S$ and $C\ge 0$ there are $N=N(s, C) \in \N$ and a finite subset $U=U(s, C) \subseteq G$ such that
\begin{equation}\label{eq:wpds-2}\{u \in G \mid \exists~n \ge N \mbox{ such that }\d(u(s), H(s))\le C, \d(ug^n(s), H(s))\le C \}\subseteq H U.
\end{equation}
\end{lem}

\begin{proof} Clearly \eqref{eq:wpds-2} implies \eqref{eq:wpds-1}, so the sufficiency in the lemma holds. To prove the necessity, suppose that $g$ is $H$-\apt{}, and
consider arbitrary $s \in S$ and $C \ge 0$.

By the assumptions, the $H$-orbit $ H(s)$ is $\eta$-quasi-convex, for some $\eta \ge 0$. Let $\ell$ be a $(\lambda,c)$-quasi-geodesic axis of $g$ based at $s$,
where $\lambda \ge 1$, $c \ge 0$,
and let $\varkappa=\varkappa(\lambda,c)\ge 0$ be the constant provided by Lemma \ref{lem:stab_of_qgeod}. Set $C'=C+\varkappa+2\delta+\eta>0$ and choose $N=N(s,C') \in \N$
and a finite subset $U=U(s,C')\subseteq G$ according to Definition \ref{defn:rwpd}.

Suppose that an element $u \in G$ satisfies
\begin{equation*}
\d(u(s),f(s))\le C~\mbox{ and }\d(ug^n(s), H(s))\le C, \mbox{ for some } f \in H \mbox{ and some } n \ge N.
\end{equation*}
Note that, without loss of generality, we can replace $u$ by $f^{-1}u$, because $f^{-1} H= H$ and $fHU=HU$. Hence we can further assume that
\begin{equation}\label{eq:dist}
\d(u(s),s)\le C~\mbox{ and }\d(ug^n(s),h(s))\le C, \mbox{ for some } h \in H \mbox{ and some } n \ge N.
\end{equation}
Let $p=u(\ell_{[s,g^n(s)]})$. Note that $p$ is $(\lambda,c)$-quasi-geodesic, $p_-=u(s)$, $p_+=ug^n(s)$ and $ug^N(s) \in p$. Therefore, by Lemma \ref{lem:quadr-H},
$\d(ug^N(s), H(s)) \le \varkappa+2\delta+C+\eta=C'$.
Since we also have that $\d(s,u(s))\le C'$, we see that $u$ belongs to the set $\{u \in G \mid \d(s,u(s))\le C', \d(ug^N(s), H(s))\le C' \}$, which is contained in $HU$
by the definition of $N$ and $U$. Hence \eqref{eq:wpds-2} holds, and the proof is finished.
\end{proof}

\begin{prop}\label{prop:rwpds} Let $H$ be a convex cocompact subgroup of $G$ and let  $g\in G$ be a loxodromic WPD element. Then  for all $s\in S$ and $C\ge 0$ there exists $N\in\N$  satisfying
\begin{equation}\label{eq:wpds-cc}
|\{u \in G \mid \exists~n \ge N \mbox{ such that }\d(u(s),s)\le C, \d(ug^n(s), H(s))\le C \}|<\infty.
\end{equation}
In particular, every loxodromic WPD element $g \in G$ is $H$-\apt{}.
\end{prop}

\begin{proof}

Choose any $s \in S$, any $C\ge 0$, and assume that $ H(s)$ is $\eta$-quasi-convex in $S$, for some $\eta \ge 0$.
Let $\ell$ be a $(\lambda,c)$-quasi-geodesic axis of $g$ based at $s$, and let $\varkappa=\varkappa(\lambda,c)\ge 0$ be the constant given by Lemma \ref{lem:stab_of_qgeod}.

Now, set $\e=2\varkappa+4\delta+ 2C+2\eta$ and
choose $N=N(s,\e)$ according to the definition of the WPD property for $g$ (\cref{defn:wpd}). 
Arguing by contradiction, assume that there exists an infinite sequence of pairwise distinct elements  ${u_1,u_2,\dots} \in G$
such that for each $i \in \N$ there are $h_i \in H$ and $n_i \in \N$ satisfying $n_i \ge N$ and $\d(u_i(s),s)\le C$, $\d(u_ig^{n_i}(s),h_i(s))\le C$.
 
For each $i \in \N$, we can consider geodesic paths $q_{1i}$ and $q_{2i}$ in $S$ joining $u_i(s)$ with $u_ig^{n_i}(s)$, and
$s$ with $h_i(s)$ respectively.
Let $\ell_i$ be the segment of $\ell$ from $s$ to $g^{n_i}(s)$. Then $u_i(\ell_i)$ is a $(\lambda,c)$-quasi-geodesic with same endpoints as the geodesic
side $q_{1i}$. So, by Lemma \ref{lem:quadr-H} there is
$f_i \in H$ such that
\begin{equation}\label{eq:u_ig}
\d(u_ig^N(s), f_i(s)) \le \varkappa+2\delta+C+\eta.
\end{equation}

Observe that for every $i \in \N$ we can estimate
\[\d(s,f_i(s)) \le \d(s,u_i(s))+\d(u_i(s),u_ig^N(s))+\d(u_ig^N(s), f_i(s))\le \varkappa+2\delta+2C+\eta+\d(s,g^N(s)),\]
where the right-hand side is independent of $i$. Consequently, recalling the assumption that $H$ is convex cocompact, we can conclude that
$\{f_i \mid i \in \N\}$ is a finite subset of $H$. Therefore there must exist $f \in H$ and
an infinite subsequence of indices $(i_j)_{j\in \N}$ such that $f_{i_j}=f$ for all $j\in \N$.

Hence, in view of \eqref{eq:u_ig}, for all $j \ge 2$ we have \[\d(u_{i_1}(s),u_{i_j}(s)) \le \d(u_{i_1}(s),s)+\d(s,u_{i_j}(s))\le 2C \le \e, \mbox{ and }\]
\[\d(u_{i_1}g^N(s),u_{i_j}g^N(s))\le \d(u_{i_1}g^N(s),f(s))+\d(f(s),u_{i_j}g^N(s))\le 2(\varkappa+2\delta+C+\eta)= \e.\]

Thus $\d(u_{i_1}^{-1}u_{i_j}(s),s) \le \e$ and $\d(u_{i_1}^{-1}u_{i_j}g^N(s),g^N(s))\le \e$, so, according to the WPD property of $g$ and the choice of $N$, the subset
$\{u_{i_1}^{-1}u_{i_j} \mid j \in \N, j \ge 2\} \subseteq G$ must be finite. Obviously the latter gives a contradiction with the assumption that the elements $u_{i_1},u_{i_2},\dots$ are pairwise distinct. 
\end{proof}

\begin{rem} \label{rem:rwpds->wpds_for_cc}
In fact, for  a convex cocompact subgroup $H$ \eqref{eq:wpds-cc} is equivalent to $g$ being $H$-\apt{}.
\end{rem}
Indeed, if $g$ is $H$-\apt{} and $H$-orbits are quasi-convex in $S$, according to Lemma \ref{lem:equiv_rwpds}, for all $C \ge 0 $ and $s \in S$ there is $N \in \N$ and a finite subset $U \subseteq G$ such that
\[V=\{u \in G \mid \exists~n \ge N \mbox{ such that }\d(u(s),s)\le C, \d(ug^n(s), H(s))\le C \}\subseteq HU.\]
But from convex cocompactness of $H$, we know that for each $u \in G$ there are only finitely many $h \in H$ such that $\d(hu(s),s)\le C$. Therefore
$|V \cap Hu|<\infty$ for every $u \in U$, and hence $|V|<\infty$, i.e., \eqref{eq:wpds-cc} holds.

Now that we have established the existence of $H$-\apt{} elements when $H$ is convex cocompact, our next goal is to use $H$-\apt{} elements to construct elements that are transverse to $H$ (see Definition \ref{def:transv_elem}). The next few lemmas will lay the groundwork for showing transversality.

\begin{lem}\label{lem:w_height} 
Let $H$ be a subgroup of $G$ with quasi-convex orbits and let  $g\in G$ be an $H$-\apt{} loxodromic element. Then there is a finite subset $U_0 \subseteq G$ such that for any $u \in G$ if $u^{-1}Hu \cap \langle g \rangle\neq \{1\}$ then $u \in HU_0$. In particular, only finitely many distinct conjugates of $H$ can intersect $\langle g \rangle$ non-trivially.
\end{lem}

\begin{proof} Fix any $s \in S$ and suppose that the orbit $ H(s)$ is $\eta$-quasi-convex in $S$, for some $\eta \ge 0$. Let $\ell$ be a $(\lambda,c)$-quasi-geodesic axis of $g$ based at $s$,
and let $\varkappa=\varkappa(\lambda,c)$ be the constant from the claim of Lemma \ref{lem:stab_of_qgeod}. Finally, we set  $C=\varkappa+2\delta+\eta$ and let
$N=N(s,C) \in\N$ be the number and $U_0=U(s,C) \subseteq G$ be the finite subset given by  Definition \ref{defn:rwpd}.

Consider any element $u \in G$ such that $g^m \in u^{-1}Hu$, for some $m \in \N$, and let $\e=\d(u(s),s)$ and $\e'=\e+\varkappa+2\delta$.

Since $g$ is loxodromic, there is $L \in \N$ such that $mL \ge N$ and
\begin{equation}\label{eq:g^mL-far}
\begin{multlined}
\d(g^{mL}(s),g^{2mL}(s))=\d(s,g^{mL}(s))>\e',~\d(s,g^{mL+N}(s))>\e' \mbox{ and } \\
\d(g^{mL+N}(s),g^{2mL}(s))=\d(g^N(s),g^{mL}(s))>\e'.
\end{multlined}
\end{equation}
Let $q=u(\ell_{[s,g^{2mL}(s)]})$ be the translate of the segment of $\ell$ from $s$ to $g^{2mL}(s)$. Then $q$ is $(\lambda,c)$-quasi-geodesic
and $ug^{mL}(s), ug^{mL+N}(s) \in q$. Moreover, $\d(q_-,s)=\e$ and, since $h=ug^{2mL}u^{-1} \in H$, we see that
\[\d(q_+, H(s)) = \d(ug^{2mL}(s),  H(s))=\d(hu(s), H(s))=\d(u(s), H(s)) \le \d(u(s),s)= \e.\]
Evidently, \eqref{eq:g^mL-far} shows that both $ug^{mL}(s)$ and $ug^{mL+N}(s)$ are situated on $q$ more than $\e'$ away from its endpoints, hence we
 can apply Lemma \ref{lem:quadr-H} to find $h_1,h_2 \in H$ such that $\d(ug^{mL}(s),h_1(s)) \le C$ and $\d(ug^{mL+N}(s),h_2(s)) \le C$.

Therefore, for $v=h_1^{-1}ug^{mL}$, we have $\d(v(s),s) \le C$ and $\d(vg^N(s), H(s)) \le C$. Hence $v \in HU_0$, according to our choice of $N$ and $U_0$.
It remains to observe that $ug^{mL}=h_3u$ for some $h_3 \in H$, as $ug^{mL}u^{-1} \in H$. Thus $u=h_3^{-1}h_1v \in HU_0$, and the lemma is proved.
\end{proof}

\begin{lem} \label{lem:inter_conj} Let $H$ be a subgroup of $G$ with quasi-convex orbits and let  $g\in G$ be an $H$-\apt{} loxodromic element.
Then for all $s \in S$ and  $D>0$ there exists a constant $M \in \N$ such that the following holds.
If $u \in G$ and $n \in \Z$ satisfy  $\d(u(s),  H(s)) \le D$, $\d(ug^n(s), H(s))\le D$ and $|n| \ge M$, then $u^{-1}Hu \cap \langle g \rangle \neq \{1\}$.
In particular, $u \in HU_0$, where $U_0$ is the finite subset of $G$ given by Lemma \ref{lem:w_height}.
\end{lem}

\begin{proof} Fix any $s \in S$ and $D>0$. Let  the orbit $ H(s)$ be $\eta$-quasi-convex in $S$, for some $\eta \ge 0$, let $\ell$ be a
$(\lambda,c)$-quasi-geodesic axis of $g$ based at $s$,
$\lambda \ge 1$, $c \ge 0$, and let $\varkappa=\varkappa(\lambda,c)\ge 0$ be the constant given by Lemma \ref{lem:stab_of_qgeod}.
Set $C=\varkappa+2\delta+D+\eta$ and let $N=N(s,C)\in \N$ and the finite subset $U=U(s,C) \subseteq G$ be provided by Lemma \ref{lem:equiv_rwpds}. Choose
$M=N+|U| \in \N$ and suppose that $\d(u(s),  H(s)) \le D$, $\d(ug^n(s), H(s))\le D$, for some $u \in G$ and $n \in \Z$ with $|n| \ge M$.

Note that if $n <0$, then for $u'=ug^n$ we have $\d(u'(s),  H(s)) \le D$, $\d(u'g^{-n}(s), H(s))\le D$ and $-n \ge M$, and the conclusion
$u'^{-1} H u' \cap \langle g \rangle \neq \{1\}$ would immediately yield $u^{-1}Hu \cap \langle g \rangle \neq \{1\}$, as $g^n\langle g \rangle g^{-n}=\langle g \rangle$.
Therefore we can further assume that
\begin{equation} \label{eq:inter_conj-1}
\d(u(s),  H(s)) \le D \text{ and } \d(ug^n(s), H(s))\le D, \mbox{ for some } u \in G \mbox{ and }  n \ge M.
\end{equation}

Let $q=u(\ell[s,g^n(s)])$ be the translate of a segment of $\ell$, so that $q_-=u(s)$ and $q_+=ug^n(s)$. For each
$i=0,1,\dots,|U|$, let $u_i=ug^i \in G$. Then  $u_i(s) \in q$, so $\d(u_i(s), H(s)) \le \varkappa+2\delta+D+\eta=C$, by Lemma~\ref{lem:quadr-H}.
On the other hand, $\d(u_ig^{n-i}(s), H(s)) \le D \le C$ by \eqref{eq:inter_conj-1}, where $n-i \ge M-i \ge N$. Thus, according to Lemma \ref{lem:equiv_rwpds}, $u_i \in HU$, for each
$i=0,1,\dots,|U|$. It follows that there exist indices $i,j$, $0\le i<j\le |U|$, such that $u_j \in Hu_i$. After recalling that $u_i=ug^i$ and $u_j=ug^j$, we get
$ug^{j-i}u^{-1}=u_ju_i^{-1} \in H$. Since $j-i>0$ and $g$ has infinite order we can conclude that $\langle g \rangle \cap u^{-1}Hu \neq \{1\}$, as required.

The last claim of the lemma is an immediate consequence of Lemma \ref{lem:w_height}.
\end{proof}

%%%%%%%%%%%%%%%%%%%%%%%%%%%%%%%%%%%%%%%%%%%%%%%%%%%%%

\begin{lem} \label{lem:gr_prod-bound}  Let $H$ be a subgroup of $G$ with quasi-convex orbits and let  $g\in G$ be an $H$-\apt{} loxodromic element satisfying $ H \cap \langle g \rangle=\{1\}$.
 Then for each $s \in S$ there is a constant $A \ge 0$ such that
$\bigl( g^k(s) \,|\,h(s) \bigr)_s\le A$, for all $k \in \Z$ and $h \in H$.
\end{lem}

\begin{proof} Suppose, on the contrary, that $\sup\left\{\bigl( g^k(s) \,|\,h(s) \bigr)_s \mid k \in \Z, h \in H\right\}=\infty$ for some $s \in S$.
Then, after replacing $g$ with $g^{-1}$
if necessary (which can be done by Lemma~\ref{lem:apt-inverse}), we can assume that
\begin{equation}\label{eq:sup-inf}
\sup\left\{\left.\bigl(\bigl. g^k(s) \,\bigr| \, h(s)\bigr)_s \,\right|\, k \in \N, h \in H\right\}=\infty.
\end{equation}

By the assumptions, the $H$-orbit $ H(s)$ is $\eta$-quasi-convex, for some $\eta \ge 0$, and $g$ has a $(\lambda,c)$-quasi-geodesic axis $\ell$, based at $s$,
for some $\lambda \ge 1$ and $c \ge 0$. Let $\varkappa=\varkappa(\lambda,c)$ be the constant given by Lemma \ref{lem:stab_of_qgeod}.

Take any $i \in \N$. The hypothesis \eqref{eq:sup-inf} implies that  there are $k \in \N$ and $h\in H$ such that  $k \ge i$ and
$\left(g^k(s)|h(s)\right)_s \ge \d(s,g^i(s))+\varkappa$. Let $q_1$ and $q_2$ be some geodesics from $s$ to $g^k(s)$ and from $s$ to $h(s)$ in $S$, respectively.
Since $g^i(s)$ belongs to the $(\lambda,c)$-quasi-geodesic path $\ell_{[s,g^k(s)]}$, which has the same endpoints as $q_1$, we know that $\d(g^i(s),x) \le \varkappa$
for some $x \in q_1$. Therefore  \[\d(s,x) \le \d(s,g^i(s))+\d(g^i(s),x)\le \bigl(\bigl. g^k(s)\,\bigr| \, h(s)\bigr)_s.\]
It follows that $\d(x,q_2) \le 6\delta$, because  any geodesic triangle with vertices $s$, $g^k(s)$ and $h(s)$ is \emph{$6\delta$-thin} in $S$
(see the proof of the implication (1)$\Rightarrow$(2) in \cite[Proposition~2.1]{Mih}). After recalling that $q_2$ is contained in the $\eta$-neighborhood of $ H(s)$,
we can conclude that
\begin{equation*}\label{eq:g-gamma}
\d(g^i(s), H(s)) \le D, ~\mbox{ for all } i \in \N, \mbox{ where } D=\varkappa+6\delta+\eta.
\end{equation*}

Therefore we can apply Lemma \ref{lem:inter_conj} (in the case $u=1$) to deduce that $H \cap \langle g \rangle \neq \{1\}$, which contradicts our assumption.
Thus $\sup\left\{\bigl( g^k(s) \,|\,h(s) \bigr)_s \mid k \in \Z, h \in H\right\}<\infty$, and  the lemma is proved.
\end{proof}

\begin{lem}\label{lem:gr_prod-spec_case} Let $g \in G$ be a loxodromic WPD element, $a \in G \setminus E_G(g)$ and $s \in S$. Then there is a constant $B \ge 0$ such that
$\bigl( a^{-1}g^{-n}(s) \,| \,g^n a(s)\bigr)_s\le B$, for all $n \in \Z$.
\end{lem}

\begin{proof} Since $g$ is loxodromic, the same is true for the element $a^{-1}ga \in G$. It follows that the subgroup $H=\langle a^{-1}ga \rangle$ is convex cocompact, and so the element $g$ is $H$-\apt{} by \cref{prop:rwpds}.

Since $a \notin E_G(g)$, the definition of $E_G(g)$ given in \eqref{eq:def_of_elem} implies that $H \cap \langle g \rangle=\{1\}$.
Therefore we can use Lemma \ref{lem:gr_prod-bound} to find $A \ge 0$ such that
\begin{equation}\label{eq:prod-bound-A}
\bigl( \bigl. g^k(s)\,\bigr| \,a^{-1}g^l a(s) \bigr)_s\le A,~\mbox{ for all }k,l \in \Z.
\end{equation}

Recall that, according to \eqref{eq:Gr-prod-def}, for any $n \in \Z$ we have
\begin{equation}\label{eq:prod-bound-C}
\bigl(\bigl. a^{-1}g^{-n}(s)\, \bigr|\,g^n a(s)\bigr)_s=\frac12 \Bigl(\d\bigl(a^{-1}g^{-n}(s),s\bigr)+\d\bigl(g^n a(s),s\bigr)-\d\bigl(a^{-1}g^{-n}(s),g^n a(s)\bigr) \Bigr).
\end{equation}
We let $C=\d(s,a(s))$ and apply the triangle inequality to obtain the following estimate:
\[ \d\bigl(a^{-1}g^{-n}(s),s\bigr) \le \d\bigl(a^{-1}g^{-n}(s),a^{-1}g^{-n}a(s)\bigr)+\d\bigl(a^{-1}g^{-n}a(s),s\bigr)=\d\bigl(a^{-1}g^{-n}a(s),s\bigr)+C.
\]
Similarly, we deduce that $\d\bigl(g^n a(s),s\bigr)\le \d\bigl(g^n (s),s\bigr)+C$ and
$\d\bigl(a^{-1}g^{-n}(s),g^n a(s)\bigr) \ge \d\bigl(a^{-1}g^{-n}a(s),g^n (s)\bigr)-2C$.
Combining these inequalities with \eqref{eq:prod-bound-C} and \eqref{eq:prod-bound-A}, we get
\begin{align*}\bigl(\bigl. a^{-1}g^{-n}(s)\, \bigr|\,g^n a(s)\bigr)_s &\le\frac12 \Bigl(\d\bigl(a^{-1}g^{-n}a(s),s\bigr)+\d\bigl(g^n(s),s\bigr)-
\d\bigl(a^{-1}g^{-n}a(s),g^n (s)\bigr)+4C \Bigr)\\
         &=\bigl(\bigl. g^{n}(s)\, \bigr|\, a^{-1}g^{-n}a(s)\bigr)_s +2C \le A+2C.
\end{align*}
Since the above is true for every $n \in \Z$, the claim of the lemma holds for $B=A+2C$.
\end{proof}

%%%%%%%%%%%%%%%%%%%%%%%%%%%%%%%%%%%%%%%%%%%%%%%%%%%%%%%%%%%%%%%%%%%%%%%%%%%
\section{Transverse elements}\label{sec:trans}

\begin{defn} \label{def:transv_elem}
Let $G$ be a group acting on a metric space $(S,\d)$ and let $s \in S$ be any point.
We will say that an element $f \in G$ is \emph{transverse to a subset $Y\subseteq  G$} (with respect to the
given action) if for every $E \ge  0$ there exists $K\in\N$ such that for all $v \in G$ we have
\begin{equation}\label{eq:trans}
\left|\{m \in \Z \mid \d(f^m(s),v Y(s)) \le E\}   \right|\leq K.
\end{equation}
\end{defn}

It is easy to see that the above definition does not depend on the choice of $s \in S$. We will always be applying this definition in the setting where $S$ is a hyperbolic metric space, $f$ is a loxodromic element and $Y$ has quasi-convex orbits. In this situation, the following lemma shows that Definition \ref{def:transv_elem} is equivalent to \cite[Definition 2.7]{AH21}.

\begin{lem}\label{lem:transverse-equiv_def}
Let $G$ be a group acting on a hyperbolic metric space $(S,\d)$, let $Y\subseteq G$, and let $s \in S$. Let $f\in G$ be a loxodromic element and let $\ell$ be a quasi-geodesic axis for $f$ based at $s$. Suppose that $Y(s)$ is quasi-convex in $S$. Then $f$ is transverse to $Y$ if and only if for all $E\geq 0$ there exists $M\geq 0$ such that for every $v\in G$,
\begin{equation}\label{eq:trans_diam_bound}
\mathrm{Diam}\left(\{x\in \ell \mid \d(x, v Y(s)) \le E\}\right)\leq M.
\end{equation}
\end{lem}
\begin{proof}
 Fix constants $\delta$, $\lambda$, $c$ and $\eta$ such that $S$ is $\delta$-hyperbolic, $\ell$ is a $(\lambda, c)$ quasi-geodesic,  $Y(s)$ is $\eta$-quasi-convex, and set $L=\d(s, f(s))>0$. 
 
Suppose that $f$ is transverse to $Y$ in the sense of Definition~\ref{def:transv_elem}. Let $E\geq 0$, $v\in G$ and suppose that $x$ and $y$ are points on $\ell$ such that $\d(x, v Y(s)) \le E$ and $\d(y, v Y(s)) \le E$. Then for some integers $m$ and $n$, $\d(x, f^m(s))\leq L$ and $\d(y, f^n(s))\leq L$. Hence we have that $\d(f^m(s), v Y(s))\leq L+E$ and $\d(f^n(s), v Y(s))\leq L+E$. 
Let $p$ be a geodesic connecting $f^m(s)$ with $f^n(s)$ in $S$. If $i$ is any integer between $m$ and $n$, then $\d(f^i(s),p) \le\varkappa$, where $\varkappa \ge 0$ is the constant from Lemma~\ref{lem:stab_of_qgeod}.
Since $v Y(s)$ is $\eta$-quasi-convex and geodesic quadrilaterals in $S$ are $2\delta$-slim, $p$ is contained in the $(\eta+2\delta+L+E)$-neighborhood of $vY(s)$. Therefore, $\d(f^i(s), v Y(s))\leq E'$, where $E'=\eta+2\delta+L+E+\varkappa$. By Definition \ref{def:transv_elem}, there is a constant $K'$ such that 
\[\left|\{i \in \Z \mid \d(f^i(s),v Y(s)) \le E'\}   \right|\leq K',\] which implies that $|m-n| <K'$. This means that $\d(f^n(s), f^m(s))\leq LK'$, and hence $\d(x, y)\leq LK'+2L$.
 
 This completes one direction of the proof. For the other direction, we note that if $\left|\{m \in \Z \mid \d(f^m(s),v Y(s)) \le E\}   \right|\ge K$, then there are integers $m$ and $n$ with $\d(f^m(s),v Y(s)) \le E$, $\d(f^n(s),v Y(s)) \le E$ and $|m-n|\ge  K-1$. This implies that $\d(f^n(s), f^m(s))\geq \frac{1}{\lambda}(L(K-1)-c)$, hence  \[\mathrm{Diam}\left(\{x\in \ell \mid \d(x, v Y(s)) \le E\}\right)\geq \frac{1}{\lambda}(L(K-1)-c). \qedhere\]
\end{proof}

Here are a few basic properties of this notion of transversality.
\begin{lem}\label{lem:transuc}
Let $G$ be a group acting on a metric space $(S,\d)$, let $f, g
\in G$ and let $Y_1, Y_2$ be subsets of $G$. If $f$ is transverse to $Y_1$ and $Y_2$ then $f$ is transverse to $Y_1\cup Y_2$.

\end{lem}
\begin{proof}
Let $E>0$, and let $K$ be a constant such that for all $\nu\in G$,
\[
\left|\{m \in \Z \mid \d(f^m(s),v Y_1(s)) \le E\}   \right|\leq K
\]
and
\[
\left|\{m \in \Z \mid \d(f^m(s),v Y_2(s)) \le E\}   \right|\leq K.
\]

Now fix $\nu\in G$, and observe that if $\d(f^m(s), v(Y_1\cup Y_2)(s)) \le E$, then either $\d(f^m(s),v Y_1(s)) \le E$ or $\d(f^m(s),v Y_2(s)) \le E$. Hence, 

\[
\left|\{m \in \Z \mid \d(f^m(s),v(Y_1\cup Y_2)(s)) \le E\}   \right|\leq 2K.
\]
Therefore, $f$ is transverse to $Y_1\cup Y_2$.
\end{proof}

For the remainder of this section we assume that $(S,\d)$ is a $\delta$-hyperbolic geodesic metric space, for some $\delta \ge 0$, and $G$ is a group acting on $S$ by isometries. 
The following observation is a consequence of Lemma~\ref{lem:inter_conj}.

\begin{lem}\label{lem:transversal-equiv}
 Let $H$ be a subgroup of $G$ with quasi-convex orbits and let  $f\in G$ be an $H$-\apt{} loxodromic element. Then $f$ is transverse to $H$ if and only if $v H v^{-1} \cap \langle f \rangle =\{1\}$, for all $v \in G$.
\end{lem}

\begin{proof} Fix any $s \in S$. Arguing by the contrapositive, assume, first, that $f^n \in v H v^{-1}$, for some $n \in \N$ and some $v \in G$.
Then for all $i \in \Z$ we have
\[ \d(f^{ni}(s), v H(s)) \le \d(f^{ni}(s),vHv^{-1}(s))+\d(s,v^{-1}(s))=E,\]
where $E=\d(s,v^{-1}(s))$. Hence $f$ cannot be transverse to $H$.

Conversely, suppose that $f$ is not transverse to $H$. Then there exist $E \ge 0$ such that for every $M \in \N$ there is $v\in G$ and there are integers
$k,l \in \Z$, with $l-k \ge M$ and $\d(f^k(s),vH(s))\le E$, $\d(f^l(s),v H(s))\le E$. It follows that for $u=v^{-1}f^k \in G$, we have
$\d(u(s), H(s)) \le E$ and $\d(uf^{l-k}(s), H(s)) \le E$. Since $M$ was arbitrary,  Lemma \ref{lem:inter_conj} implies that
$u^{-1}Hu \cap \langle f \rangle \neq \{1\}$, which yields that $v H v^{-1} \cap \langle f \rangle \neq \{1\}$, as required.
\end{proof}

\begin{lem}\label{lem:new-D} Suppose that $g \in G$ is a loxodromic WPD element. Then for every $a \in G\setminus E_G(g)$ there exists $N \in \NN$ such that $g^n a$ is loxodromic WPD, for any $n \in \Z$ with $|n| \ge N$.
\end{lem}

\begin{proof} By \cite[Proposition~A.1 and Corollary A.2]{BFGS} there exists a hyperbolic geodesic metric space $T$ with a cobounded $G$-action and a $G$-equivariant coarsely Lipschitz map $\pi:S \to T$, such that $g$ acts as a loxodromic WPD isometry of $T$. Note that $E_G(g)$ is independent of the choice of $S$ by Lemma~\ref{lem:elemrem}, hence we can apply \cite[Lemma 5.4]{AMS} to find $N \in \N$ such that $g^n a$ is a loxodromic WPD element for the $G$-action on $T$, provided $|n|\ge N$. Since $\pi: S \to T$ is a $G$-equivariant coarsely Lipschitz map, it is easy to see that $g^n a$ will also be a loxodromic WPD element with respect to the action of $G$ on $S$.    
\end{proof}

\begin{prop} \label{prop:existence-transversal}  Let $H$ be a subgroup of $G$ with quasi-convex orbits and let  $g\in G$ be an $H$-\apt{} loxodromic WPD element. Let $U_0 \subseteq G$ be the subset provided by Lemma \ref{lem:w_height} and let
$a \in G\setminus \left( U_0^{-1} H U_0 \cup E_G(g)\right)$. Then the element $f=g^n a$ is loxodromic WPD and transverse to $H$ for all sufficiently large $n \in \N$.
\end{prop}

\begin{proof} Let $s \in S$ be any point, and let the orbit $ H(s)$ be $\eta$-quasi-convex, for some $\eta \ge 0$. Let $B\ge 0$ be given by Lemma \ref{lem:gr_prod-spec_case};
choose  $C_0=\max\{B, 168\delta\}$ and $C_1>12(C_0+12\delta)$, as in Lemma \ref{lem:piecewise_qgeod}. Denote $C_2=\d(s,a(s))$, $D=2\delta+\eta+2C_0+C_2$, and let
$M=M(s,D) \in \N$ be the constant provided by Lemma \ref{lem:inter_conj}.

Since the element $g \in G$ is loxodromic, there exists $N_1 \in \N$, such that for every $n \in \N$, with $n \ge N_1$, we have
\begin{equation}\label{eq:long_segm}
\d(s,g^n a(s))\ge \d(s,g^n(s))-\d(g^n(s),g^na(s)) = \d(s, g^{n}(s))-C_2 \ge C_1.
\end{equation}
On the other hand, since $a \notin E_G(g)$, by Lemma~\ref{lem:new-D} there exists $N_2 \in \N$ such that the element $g^na$ is loxodromic and WPD for all $n \in \N$, $n \ge N_2$.

Let us now fix any $n \in \N$ such that $n \ge \max\{M,N_1,N_2\}$. To finish the proof it remains to show that the element $f=g^na$ is transverse to $H$.
Arguing by contradiction, suppose that there exist $E \ge 0$ such that for all $R\in\N$, there is $v \in G$ such that  $|\{m \in \Z \mid \d(f^m(s),v H(s)) \le E\}|\geq R+1$. Then for $u=v^{-1} \in G$,
there must exist integers $k<l$ and elements $h_k,h_l \in H$ such that $\d(uf^k(s),h_k(s))\le E$, $\d(uf^l(s),h_l(s))\le E$ and $l-k\geq R$.
Since $f$ is loxodromic, by choosing $R$ sufficiently large we can ensure that there is $j \in \Z$, $k<j<l$, such that for any $\epsilon \in \{-1,0,1\}$ we have
\begin{equation}\label{eq:j}
\d(f^k(s), f^{j+\epsilon}(s))>E+2\delta+2C_0 ~\mbox{ and }~ \d(f^{j+\epsilon}(s),f^l(s))>E+2\delta+2C_0.
\end{equation}
Observe that, in view of \eqref{eq:long_segm}, we have $\d(uf^{i-1}(s),uf^i(s))=\d(s,g^na(s))\ge C_1$,
for all $i \in \Z$, and
\[\bigl(uf^{i-1}(s) \,| \,uf^{i+1}(s) \bigr)_{uf^i(s)}=
\bigl( f^{-1}(s) \,| \, f(s) \bigr)_s \le B \le C_0,
\]
by Lemma~\ref{lem:gr_prod-spec_case}, for all $i \in \Z$.
Therefore we can apply Lemma~\ref{lem:piecewise_qgeod} to conclude that $\d(uf^i(s),q) \le 2C_0$ for every $i \in \Z$, with $k \le i \le l$, where $q$ is a geodesic in $S$ joining $uf^k(s)$ with $uf^l(s)$.

Let $x$ be a point of $q$ satisfying $d(uf^j(s),x) \le 2C_0$. Then  \eqref{eq:j} implies that $\d(x,q_-)=\d(x,uf^k(s))>E+2\delta$ and
$\d(x,q_+)=\d(x,uf^l(s))>E+2\delta$. It follows that the distances from $x$ to $[h_k(s),q_-]$ and to $[q_+,h_l(s)]$
are greater than $2\delta$ (see Figure \ref{fig:trans}). Consequently, since geodesic quadrangles in $S$ are $2\delta$-slim, $x$ must be at most $2\delta$ away
from $[h_k(s),h_l(s)]$. Hence $\d(x, H(s)) \le 2\delta+\eta$, and thus $\d(uf^j(s), H(s)) \le 2\delta+\eta+2C_0$. Similarly,
$\d(uf^{j-1}(s), H(s)) \le 2\delta+\eta+2C_0$ and $\d(uf^{j+1}(s), H(s)) \le 2\delta+\eta+2C_0$.

\begin{figure}[ht!]
  \begin{center}
   \includegraphics{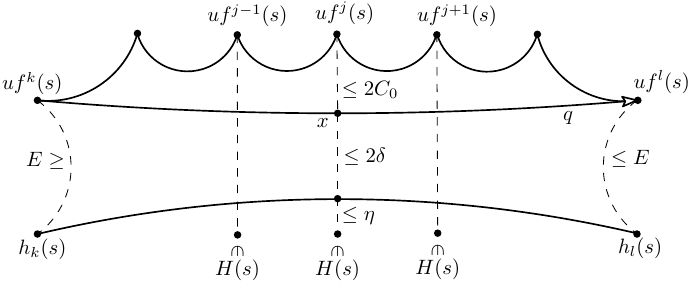}
  \end{center}
  \caption{Illustration of the proof of Proposition~\ref{prop:existence-transversal}}\label{fig:trans}
\end{figure}

Denote $u_1=uf^j \in G$ and $u_2=uf^ja^{-1}=u_1a^{-1} \in G$. Note that $\d(u_1(s), H(s)) \le D$, where $D$ is the constant defined at the beginning of the proof, and
\begin{align*}
\d(u_1g^n(s), H(s)) &\le \d(u_1g^na(s), H(s))+\d(u_1g^n(s),u_1g^na(s))\\ &=\d(u_1f(s), H(s))+\d(s,a(s))
   = \d(uf^{j+1}(s), H(s))+C_2 \le D.
\end{align*}
At the same time, we have $\d(u_2(s), H(s))\le \d(uf^j(s), H(s))+\d(s,a(s))\le D$, and
\[\d(u_2g^{-n}(s), H(s))=\d(uf^ja^{-1}g^{-n}(s), H(s)) =\d(uf^{j-1}(s), H(s)) \le D.
\]
Since $n=|-n| \ge M$, we can apply Lemma \ref{lem:inter_conj} to deduce that $u_1,u_2 \in HU_0$. This yields that
$a=u_2^{-1}u_1 \in U_0^{-1}HU_0$, contradicting the choice of $a$.
\end{proof}

\begin{cor}\label{cor:ex_of_trans}
Suppose that $H_1,\dots,H_k$ are subgroups of $G$ with quasi-convex orbits and $F \leqslant G$ contains a loxodromic WPD element $g$ which is $H_i$-\apt{}, for every $i=1,\dots,k$. If $|F:(F \cap vH_iv^{-1})|=\infty$, for all $v \in G$ and $i =1,\dots,k$, then $F$ contains a loxodromic WPD element $f$ which is transverse to $H_i$, for each $i=1,\dots,k$.
\end{cor}

\begin{proof} If $F$ is virtually cyclic then $|F:\langle g\rangle|<\infty$, so the assumptions imply that $\langle g \rangle \cap vH_iv^{-1}=\{1\}$, for all $v \in G$
and $i =1,\dots,k$. Therefore $g$ is transverse to $H_i$,  for every  $i =1,\dots,k$, by Lemma \ref{lem:transversal-equiv}. Hence we can take $f=g$.

Thus we can further suppose that $F$ is not virtually cyclic; in particular, $|F:F\cap E_G(g)|=\infty$.
Let $U_{0,i} \subseteq G$ be the finite subset provided by Lemma~\ref{lem:w_height} for $H_i$. Let us show that
$F \setminus \left(\bigcup_{i =1}^k U_{0,i}^{-1} H_i U_0 \cup E_G(g)\right)$ is non-empty. Indeed, otherwise $F$ would be covered by finitely many right cosets modulo the subgroups
$u^{-1} H_iu$, $u \in U_{0,i}$, $i=1,\dots,k$, and $E_G(g)$, which, by a well-known result of B. Neumann \cite[Lemma 4.2]{Neumann} (see also \cite[Proposition 7.3]{M-dis}),
would imply that one of these subgroups intersects $F$ in a subgroup of finite index, contradicting our assumptions.

Thus there must exist at least one element  $a \in F$, $a \notin \bigcup_{i =1}^k U_{0,i}^{-1} H_i U_{0,i} \cup E_G(g)$. This allows us to apply
Proposition \ref{prop:existence-transversal} to conclude that for all large enough $n \in \N$, $f=g^n a \in F$ is a loxodromic WPD element which is
transverse to $H_i$, for each $i=1,\dots,k$.
\end{proof}

In the case when each $H_i$ is convex cocompact, Corollary~\ref{cor:ex_of_trans} can be combined with Proposition~\ref{prop:rwpds} to achieve the following.
\begin{cor}\label{cor:cc-trans}
Suppose that $H_1,\dots,H_k$ are convex cocompact subgroups of $G$ and $F \leqslant G$ contains a loxodromic WPD element. If $|F:(F \cap vH_iv^{-1})|=\infty$, for all $v \in G$ and $i =1,\dots,k$, then $F$ contains a loxodromic WPD element $f$ which is transverse to $H_i$, for each $i=1,\dots,k$.
\end{cor}

\begin{cor}\label{cor:cc-trans_all_of_G}
Suppose that $H_1,\dots, H_k$ are infinite index convex cocompact subgroups of $G$ and $G$ has at least one loxodromic WPD element. Then $G$ contains a loxodromic WPD element $f$ which is transverse to $H_i$, for each $i=1,\dots,k$.
\end{cor}

Note that if $G$ is not a hyperbolic group, then convex cocompact subgroups are automatically of infinite index.

%%%%%%%%%%%%%%%%%%%%%%%%%%%%%%%%%%%%%%%%%%%%%%%%%%%%%%%%%%%%%%%%%%%%%%%
\section{Random walks and \texorpdfstring{$\mu$}{mu}-mixing actions}\label{sec:rw_htt}

In this section, we use random walks to prove Theorem \ref{Thm:CC}, a more general version of Theorem \ref{Thm:main1}. Along the way we state an extension of some results from \cite{AH21} which follow from Corollary~\ref{cor:cc-trans}.

\subsection{Free products with random subgroups}\label{sec:fp}
We first state one of the main results of \cite{AH21}. We will also include in the statement a result of Taylor-Tiozzo that, under the assumptions of this theorem, a random $k$--generated subgroup is isomorphic to the free group of rank $k$ asymptotically almost surely \cite[Theorem 1.2]{TayTio}.

\begin{thm}\cite[Theorem~1.1]{AH21}\label{thm:AHmain}
Let $G$ be a group with a non-elementary, partially WPD action on a hyperbolic metric space $S$ with $E(G)=1$. Let $(\mu_i)$ be a sequence of permissible probability measures on $G$ and let $H$ be a subgroup of $G$. Suppose $H$ has quasi-convex orbits in $S$ and there exists a loxodromic WPD element $f\in\bigcap_i \Gamma_{\mu_i}$ transverse to $H$. Then for every $k \in \N$, asymptotically almost surely a random $k$-generated subgroup $R$ in $G$ will be free of rank $k$, will satisfy $\langle H, R\rangle\cong H\ast R$, and $\langle H, R\rangle$ will have quasi-convex orbits in $S$. Moreover, if $H$ is quasi-isometrically embedded in $S$, then $\langle H, R\rangle$ is quasi-isometrically embedded in $S$.
\end{thm}

By the definition given in \cite[Section~2.2]{AH21}, a subgroup $H$ is \emph{quasi-isometrically embedded in $S$} if and only if $H$ is convex compact in the sense of Section~\ref{sec:hypgeo}. If we assume that $\Gamma=\bigcap_i \Gamma_{\mu_i}$ contains a loxodromic WPD element and that for all $v\in G$, $|\Gamma : (\Gamma \cap  vHv^{-1})|=\infty$, then Corollary~\ref{cor:cc-trans} implies that $\Gamma$ contains a loxodromic WPD element which is transverse to $H$. Hence, combining Theorem~\ref{thm:AHmain} and Corollary~\ref{cor:cc-trans} gives the following.

\begin{thm}\label{thm:cc-rw}
Suppose that $G$ is a group with a non-elementary, partially WPD action on a hyperbolic metric space $S$, and $E(G)=1$. Let $(\mu_i)$ be a sequence of permissible probability measures on $G$ and let $H \in \Sub_{\infty}^{cc}(G\curvearrowright S)$. Suppose  there exists a loxodromic WPD element $f\in\Gamma=\bigcap_i \Gamma_{\mu_i}$ and that for all $v\in G$, $|\Gamma : (\Gamma \cap  vHv^{-1})|=\infty$. Then for all $k\geq 1$, a random $k$-generated subgroup $R$ in $G$ will be free of rank $k$ and will satisfy $\langle H, R\rangle\cong H\ast R$ and $\langle H, R\rangle\in\Sub_{\infty}^{cc}(G\curvearrowright S)$ asymptotically almost surely. 
\end{thm}

 We can simplify the statement of the previous theorem by adding the assumption that each $\mu_i$ has generating support.

\begin{cor}\label{cor:ccrw}
Suppose that $G$ is a group with a non-elementary, partially WPD action on a hyperbolic metric space $S$, and $E(G)=1$. Let $(\mu_i)$ be a sequence of permissible probability measures with generating support on $G$, and let $H \in \Sub_{\infty}^{cc}(G\curvearrowright S)$. Then for all $k\geq 1$, a random $k$-generated subgroup $R$ in $G$  will be free of rank $k$ and will satisfy $\langle H, R\rangle\cong H\ast R$ and $\langle H, R\rangle\in\Sub_{\infty}^{cc}(G\curvearrowright S)$ asymptotically almost surely. 
\end{cor}

\begin{rem}
As in \cite[Theorem 3.3]{AH21}, if the action of $G$ on $S$ is cobounded then the condition $E(G)=\{1\}$ can be removed from the previous theorems at the cost of replacing the isomorphism $\langle H, R\rangle\cong H\ast R$ with $\langle H, R, E(G)\rangle\cong HE(G)\ast_{E(G)} RE(G)$. However, in this case if $H\cap E(G)=1$ then it follows that we again have $\langle H, R\rangle\cong H\ast R$.  
\end{rem}

Note that Theorem~\ref{thm:intro-ccrw} is a special case of Corollary \ref{cor:ccrw} obtained by adding the assumptions that $(\mu_i)$ is a constant sequence with each $\mu_i=\mu$ and $\mathrm{Supp}(\mu)$ is a finite, symmetric generating set of $G$, and then applying the conclusion of Corollary \ref{cor:ccrw} with $k=1$. The condition that $w_n$ is a loxodromic element asymptotically almost surely is given by \cite[Theorem 1.4]{MT18}.

The notion of a stable subgroup of a finitely generated group was introduced by Durham-Taylor in \cite{DurTay} as a generalization of the notion of a quasi-convex subgroup of a hyperbolic group. In right-angled Artin groups and, more generally, hierarchically hyperbolic groups, a subgroup is stable if and only if it is convex cocompact with respect to an appropriate acylindrical (hence, partially WPD) action on a hyperbolic metric space \cite{KMT17, ABD}. In particular, the following corollary is an immediate consequence of Theorem \ref{thm:cc-rw}. This essentially resolves \cite[Conjecture 6.14]{AH21} and \cite[Conjecture 6.15]{AH21}, though we have added a few additional and necessary hypothesis.

\begin{cor}
Let $G$ be a right-angled Artin group or, more generally, a hierarchically hyperbolic group. Suppose $G$ is acylindrically hyperbolic and $H$ is an infinite index stable subgroup of $G$ such that $H\cap E(G)=\{1\}$.  Let $(\mu_i)$ be a sequence of permissible probability measures on $G$ with generating supports. Then a random $k$-generated subgroup $R$ of $G$  is free of rank $k$, satisfies $\langle H, R\rangle\cong H\ast R$ and $\langle H, R\rangle$ is stable in $G$ asymptotically almost surely.
\end{cor}

\subsection{\texorpdfstring{$\mu$}{mu}-mixing actions}
We now discuss one of the main technical results from \cite{AH21} which we will use in the proof of Theorem \ref{Thm:CC}. We first recall some notation from \cite{AH21}. Let $G$ be a group acting by isometries on a metric space $(S,\d)$. If $Y\subseteq 
G$, by a word in the alphabet $Y$ we mean a product of elements of $Y$, $g_1\dots g_n$, where each $g_i\in Y$. In particular, such words contain only elements of $Y$ and not their inverses. However, in most cases we consider, the set $Y$ will be symmetric. 

Now let $W=g_1\dots g_n$ be a word in the elements of $G$. We say that $W$ \emph{represents} the element $g=g_1\cdots g_n\in G$ and $W$ is \emph{reduced} if $g_i\neq g_{i+1}^{-1}$, for $i=1,\dots,n-1$. 
Given $s \in S$, a \emph{path labeled by $W$ and based at $s$} in $S$ is a path from $s$ to $g(s)$ which is a concatenation of geodesics of the form
\[
[s, g_1(s)]\,[g_1(s), g_1g_2(s)]\ldots[g_1\cdots g_{n-1}(s), g_1\cdots g_n(s)].
\]

We will refer to $g_i$ as the \emph{label} of the subpath $[g_1\cdots g_{i-1}(s), g_1\cdots g_i(s)]$; subpaths labeled by subwords of $W$ are defined similarly. Note that since $G$ is acting by isometries, the length of $[g_1\cdots g_{i-1}(s), g_1 \cdots g_i(s)]$ is equal to $\d(s, g_i(s))$.

The following is a slight simplification of \cite[Theorem~3.2]{AH21}. Note that, in the terminology of \cite{AH21}, a word in the alphabet $Y$ is allowed to include inverses of elements of $Y$. In our terminology, this is equivalent to a word in $Y^{\pm 1}$.

\begin{thm}\label{thm:rw}
Let $G$ be a group with a non-elementary partially WPD action on a hyperbolic metric space $S$. Suppose $E(G)=\{1\}$ and $\mu$ is a permissible probability measure on $G$. Let $s\in S$ and $Y$ be a subset of $G\setminus\{1\}$  with the orbit $Y(s)$ $\sigma$--quasi-convex. Suppose also that there exists a loxodromic WPD element $f\in\Gamma_\mu$ which is transverse to $Y$. Let $w_n$ and $u_n$ be independent random walks with respect to $\mu$. Then there is a sequence of constants $(c_n)$ such that the following holds asymptotically almost surely. Given any reduced non-trivial word $W$ in the alphabet $Y\cup\{w_n^{\pm1}\}\cup\{u_n^{\pm 1}\}$ with no two consecutive letters belonging to $Y$, the path in $S$ based at $s$ and labeled by $W$ is a $(8, c_n)$--quasi-geodesic. 
\end{thm}

In the next lemma we keep the assumptions and notation used in Theorem~\ref{thm:rw}. The argument is essentially contained in the proof of \cite[Theorem~3.2]{AH21}, but we will include it here as well.

\begin{lem}\label{lem:long_words}
Let $W$ be a reduced word in the alphabet $Y\cup\{w_n^{\pm1}\}\cup\{u_n^{\pm 1}\}$ with no two consecutive letters belonging to $Y$. Suppose $W$ contains at least one letter belonging to $\{w_n^{\pm1}\}\cup\{u_n^{\pm 1}\}$ and $W$ represents $g_n\in G$. Then for all $N\geq 0$, $\d(s, g_n(s))\geq N$ asymptotically almost surely.  In particular, $g_n\neq_G 1$ asymptotically almost surely. 
\end{lem}

\begin{proof}
 Let $p$ be a path in $S$ based at $s$ and labeled by $W$. Let $D$ be the constant given by Theorem \ref{thm:drift} and $c_n$ the constants given by Theorem \ref{thm:rw}. By the ``moreover" part of \cite[Theorem~3.2]{AH21}, the constants $c_n$ can be chosen so that $\frac{c_n}{n}\to 0$ as $n\to \infty$. Hence we can choose $\e>0$ such that for all sufficiently large $n$, and $D-\e D-\frac{c_n}{n}>0$.
 
 By assumption, $p$ contains a subpath labeled by an element of $\{w_n^{\pm1}\}\cup\{u_n^{\pm 1}\}$. Without loss of generality, suppose this subpath is labeled by $w_n$. Since this subpath has length $\d(s, w_n(s))$, by Theorem \ref{thm:drift} we have
 \[
\|p\|\geq \d(s, w_n(s))\geq (1-\e)Dn,
 \]
asymptotically almost surely. Also, as $p$ is an $(8, c_n)$ quasi-geodesic (which happens asymptotically almost surely), we have 
\[
\|p\|\leq 8\d(s, g_n(s))+c_n.
\]

Combining these inequalities, we obtain
\[
\d(s, g_n(s))\geq \frac{1}{8}\left((1-\e)Dn-c_n \right)=\frac{1}{8}\left(D-\e D-\frac{c_n}{n} \right)n ,
\]
asymptotically almost surely. By our choice of $\e$, $\frac{1}{8}\left(D-\e D-\frac{c_n}{n} \right)>0$ for all sufficiently large $n$. Hence, for any fixed $N\geq 0$, $\frac{1}{8}\left(D-\e D-\e B \right)n\geq N$ for all but finitely many values of $n$. Therefore, $\d(s, g_n(s))\geq N$ asymptotically almost surely.
\end{proof}

We are now ready to prove the following, which is the main ingredient in the proof of Theorem \ref{Thm:CC}.

\begin{prop}\label{p:tt}
Let $G$ be a group with a non-elementary, partially WPD action on a hyperbolic space $S$ and $E(G)=\{1\}$. Let $H$ and $K$ be infinite index convex cocompact subgroups of $G$ and let $\mc{F}$ be a finite subset of $G$. Let $\mu$ be a permissible probability measure on $G$ with generating support and let $w_n$ be a random walk with respect to $\mu$. Then the subgroup $L_{n}=\langle w_n^{-1}Hw_n, K\rangle$ will satisfy the following conditions asymptotically almost surely.
\begin{enumerate}
\item[(a)] $L_{n}$ is convex cocompact;
\item[(b)] $L_{n}\cap \mathcal F=K\cap \mathcal F$ and $w_nL_{n}w_n^{-1}\cap \mathcal F=H\cap \mathcal F$;
\item[(c)] $L_n$  has infinite index in $G$.
\end{enumerate}
\end{prop}
\begin{proof} Fix a basepoint $s\in S$.
By Corollary \ref{cor:cc-trans_all_of_G} there exists a loxodromic WPD element $h \in G$ which is transverse to $H$ and $K$. Choose $Y=(H\cup K)\setminus\{1\}$ and note that $Y(s)$ is quasi-convex in $S$, by Lemma~\ref{lem:union_of_qc}, and $h$ is transverse to $Y$, by Lemma~\ref{lem:transuc}. By Theorem \ref{thm:rw}, we can assume that if $W$ is a word in the alphabet $Y\cup\{w_n^{\pm 1}\}$ with no consecutive letters belonging to $Y$, then the path in $S$ based at $s$ and labeled by $W$ is an $(8, c_n)$-quasi-geodesic.

(a) Let us first show that the induced action of  $L_{n}=\langle w_n^{-1}Hw_n, K\rangle \leqslant G$ on $S$ is metrically proper. Choose any $R>0$ and suppose that $g \in L_n$ satisfies $\d(s,g(s))\le R$. The element $g$ can be represented by a word 
\begin{equation}\label{eq:word_for_g}
    W=k_0 w_n^{-1} h_1 {w_n} k_1 \dots w_n^{-1} h_m {w_n}k_{m},
\end{equation}
with $m\in \N\cup \{0\}$, $h_1,\dots,h_m \in H\setminus\{1\}$, $k_0,k_m \in K$ and $k_1,\dots,k_{m-1} \in K\setminus\{1\}$. If we consider $W$ as a reduced word over the alphabet $Y \cup \{w_n^{\pm 1}\}$, then, by Theorem~\ref{thm:rw}, the path $q$ in $S$, based at $s$ and labeled by $W$, is  $(8, c_n)$--quasi-geodesic. It follows that 
\[\|q\| \le 8\, \d(s,g(s))+c_n \le 8R+c_n.\] Recall that by the definition of $q$, $\d(s,k_i(s)) \le \|q\|$, for $i=0,\dots,m$, and $\d(s,h_j(s)) \le \|q\|$, for $j=1,\dots,m$. Hence, $\d(s,k_i(s)) \le 8R+c_n$, for each $i=0,\dots,m$, and there are only finitely many such elements $k_i \in K$ because the action of $K$ on $S$ is proper. Similarly, there are only finitely many possibilities for $h_j \in H$, $j=1,\dots,m$. Therefore the element $g \in L_n$ can take only finitely many values, so the action of $L_n$ on $S$ is proper. 

Now, assume that the orbits $H(s)$ and $K(s)$ are $\eta$-quasi-convex in $S$, for some $\eta \ge 0$. Let us show that the orbit $L_n(s)$ is also quasi-convex (it would follow that every orbit of $L_n$ is quasi-convex, see Lemma~\ref{lem:qc_orbits}). Let $p$ be a geodesic path in $S$ joining two points of this orbit. After applying the translation by an element of $L_n$, we can assume that $p$ starts at $s$ and terminates at $g(s)$, for some $g \in L_n$. Let $W$ be a word as in \eqref{eq:word_for_g}, representing $g$ in $G$ and considered as a reduced word over the alphabet $Y \cup \{w_n^{\pm 1}\}$. Let $q$ be the path in $S$ based at $s$ and labeled by $W$. Then $q$ is $(8, c_n)$-quasi-geodesic, so for every points $x$ of $p$ there is a point $y$ of $q$ such that $\d(x,y) \le \varkappa_n$, where $\varkappa_n \ge 0$ is the constant provided by Lemma~\ref{lem:stab_of_qgeod}. By construction, the path $q$ is a concatenation of geodesic subpaths labeled by $w_n^{\pm 1}$, $k_i$  or $h_j$, and $y$ must belong to one of them.  

Suppose, first, that $y$ lies on a subpath $r$, of $q$, labeled by $k_i$, for some $i=0,\dots,m$:
\[r=[k_0 w_n^{-1} h_1 {w_n} k_1 \dots w_n^{-1} h_i {w_n}(s),k_0 w_n^{-1} h_1 {w_n} k_1 \dots w_n^{-1} h_i {w_n}k_{i}(s)]=u([s,k_i(s)]),\]
where $u=k_0 w_n^{-1} h_1 {w_n} k_1 \dots w_n^{-1} h_i {w_n} \in L_n$. Since the orbit $K(s)$ is $\eta$-quasi-convex,  we have $[s,k_i(s)] \subseteq \mathcal{N}_\eta(K(s))$, whence
\[y \in r \subseteq \mathcal{N}_\eta(uK(s)) \subseteq \mathcal{N}_\eta(L_n(s)).\]

If $y$ lies on a subpath labeled by $w_n$ then it is within $d_n=\d(s,w_n(s))$-distance from an element of the form 
$k_0 w_n^{-1} h_1 {w_n} k_1...w_n^{-1} h_i {w_n} \in L_n$, for some $i=0,\dots,m$. Similarly, $y \in \mathcal{N}_{d_n}(L_n)$ when $y$ belongs to a subpath of $q$ labeled by $w_n^{-1}$. The final possibility is that $y$ belongs to a geodesic subpath $r$, of $q$, labeled by $h_j$, for some $j=1,\dots,m$, so that
\[r=v([s,h_j(s)]), \text{ where } v=k_0 w_n^{-1} h_1 {w_n} k_1\dots k_{j-1}w_n^{-1}.\] Since $[s,h_j(s)] \subseteq \mathcal{N}_\eta(H(s))$, we have
\[y \in \mathcal{N}_\eta(vH(s)) \subseteq \mathcal{N}_{\eta+d_n}(vHw_n(s))
\subseteq \mathcal{N}_{\eta+d_n}(L_n(s)),\]
because $vHw_n=k_0 w_n^{-1} h_1 {w_n} k_1\dots k_{j-1}w_n^{-1}Hw_n \subseteq L_n$. 

Thus we conclude that $x \in \mathcal{N}_{\varkappa_n+\eta+d_n}(L_n(s))$, i.e., $L_n(s)$ is $(\varkappa_n+\eta+d_n)$-quasi-convex.

(b)  Let $N=\max\{\d(s, f(s))\mid f\in \mathcal F\}+1$. If $g \in L_n \setminus K$ is any element, then any reduced word $W$ over the alphabet $Y \cup \{w_n^{\pm 1}\}$ (as in  \eqref{eq:word_for_g}), representing $g$, must contain at least one instance of $w_n$. By Lemma~\ref{lem:long_words}, the latter implies that $\d(s, g(s))\ge N$, so that $g\notin\mathcal F$ asymptotically almost surely. Thus $L_n \cap \mathcal{F}=K \cap \mathcal{F}$. The second equality in (b) can be proved similarly.

(c) It remains to show that $L_{n}$ has infinite index in $G$. If $L_n$ is trivial this is obvious because $G$ contains an element of infinite order, so assume that $L_n \neq \{1\}$. Let $u_n$ be another independent random walk with respect to $\mu$ and consider the subgroup $\langle L_{n}, u_n\rangle$. By Lemma~\ref{lem:long_words}, if $W$ is a reduced word in $Y\cup\{w_n^{\pm 1}\}\cup\{u_n^{\pm 1}\}$, with no two consecutive letters belonging to $Y$ and containing at least one occurrence of $u_n^{\pm 1}$, then $W$ represents a non-trivial element of $G$, in particular, $u_n \neq 1$. This implies that $\langle L_{n}, u_n\rangle\cong L_{n}\ast\langle u_n\rangle$ in $G$. Since $L_{n}$ is non-trivial, it must have infinite index in this free product, hence $|G:L_n|=\infty$.
\end{proof}

\begin{thm}\label{Thm:CC}
Let $G$ be a group with a non-elementary  partially WPD action on a hyperbolic metric space $S$ and suppose that $G$ has no non-trivial finite normal subgroups. Let $\mu$ be a permissible measure on $G$ with generating support. Then the action of $G$ on the closure of $\Sub_\infty^{cc}(G\curvearrowright S)$ in $\Sub(G)$ is topologically $\mu$-mixing.
\end{thm}

\begin{proof}
 Suppose $U$, $V$ are open subsets of the closure of $\Sub_\infty^{cc}(G \curvearrowright S)$ and $K\in U\cap \Sub_\infty^{cc}(G \curvearrowright S)$, $H\in V\cap \Sub_\infty^{cc}(G \curvearrowright S)$. By the definition of the topology on $\Sub_\infty^{cc}(G \curvearrowright S)$, there exists a finite subset $\mathcal F\subset G$ such that for any subgroup $J$, if $J\cap\mathcal F=K\cap\mathcal F$ then $J\in U$ and if $J\cap \mathcal F=H\cap\mathcal F$ then $J\in V$. Now let $w_n$ be a random walk with respect to $\mu$ and let $L_{n}$ be given by Proposition~\ref{p:tt}. Then, asymptotically almost surely, $L_{n}$ belongs to $\Sub_\infty^{cc}(G \curvearrowright S)$, $L_{n}\in U$ and $w_nL_nw_n^{-1}\in V$. This means that $w_n(U)\cap V\neq \emptyset$ asymptotically almost surely. Thus, the action of $G$ on the closure of $\Sub_\infty^{cc}(G \curvearrowright S)$ in $\Sub(G)$ is topologically $\mu$-mixing.
\end{proof}

%%%%%%%%%%%%%%%%%%%%%%%%%%%%%%%%%%%%%%%%%%%%%%%%%%%%%%%%%%%%%%

\section{Relatively quasi-convex subgroups of relatively hyperbolic groups}\label{sec:relhyp}
Here we explain the proof of Theorem \ref{Thm:main2}, that is we show that the action of a relatively hyperbolic group on the closure of its space of infinite index relatively quasi-convex subgroups in $\Sub (G)$ is topologically $\mu$-mixing. In the setting of a relatively hyperbolic group acting on its relative Cayley graph, our notion of a convex cocompact subgroup is equivalent to the notion of a \emph{strongly relatively quasi-convex subgroup} by \cite[Theorem~4.13]{Osi06}. However, there may also be relatively quasi-convex subgroups which are not convex cocompact, for example any infinite peripheral subgroup will be like this. Hence, Theorem~\ref{Thm:main2} does not follow directly from Theorem~\ref{Thm:main1}, though the proofs of these two results are quite similar in spirit.

After establishing some definitions and notation, we will show the existence of a loxodromic, WPD element transverse to any finite collection of infinite index relatively quasi-convex subgroups (Corollary \ref{cor:rh_trans}). For a single infinite index relatively quasi-convex subgroup, such an element was shown to exist in \cite[Proposition 6.4]{AH21}.

The following definition of relative hyperbolicity is due to Bowditch \cite{Bow}. Here a graph is called \emph{fine} if for all $n$, the graph contains finitely many circuits of length $n$, where a circuit is a cycle without self-intersection. 

\begin{defn}\label{defn:relhyp}
 $G$ is hyperbolic relative to the family of subgroups $\{H_\lambda\}_{\lambda \in \Lambda}$ if $G$ acts on a fine, hyperbolic graph $K$ with finite edge stabilizers, finitely many orbits of edges, and each vertex stabilizers is either finite or is conjugate to exactly one element of  $\{H_\lambda\}_{\lambda \in \Lambda}$. 
\end{defn}

For countable groups, this definition is equivalent to several others in the literature, see \cite{Hru}. In particular, it is equivalent to the definitions used by Osin in \cite{Osi06} and \cite{Osi13} which prove the following results respectively.

Let $G$ be a group and let $\{H_\lambda\}_{\lambda \in \Lambda}$ be a family of subgroups of $G$. Assume that there exists a finite subset $X \subseteq G$ such that $G$ is generated by $X \cup \bigcup_{\lambda \in \Lambda} H_\lambda$. In this case $X$ is called a \emph{relative generating set} of $G$, and the Cayley graph $\Gamma(G,X\cup\mathcal{H})$, where $\mathcal{H}=\sqcup_{\lambda \in \Lambda} H_\lambda \setminus\{1\}$, is called a \emph{relative Cayley graph} of $G$.

\begin{thm}\cite[Corollary 2.54]{Osi06}
Suppose that a group $G$ is  hyperbolic relative to a collection of subgroups $\{H_\lambda\}_{\lambda \in \Lambda}$ and $X$ is a finite relative generating set of $G$. Then the relative Cayley graph $\Gamma(G,X\cup\mathcal{H})$, equipped with the standard edge-path metric $\d$, is a hyperbolic metric space.
\end{thm}

\begin{thm}\cite[Proposition 5.2]{Osi13}\label{thm:rh_acy}
Suppose that a group $G$ is  hyperbolic relative to a collection of subgroups $\{H_\lambda\}_{\lambda \in \Lambda}$ and $X$ is a finite relative generating set of $G$. Then the action of $G$ on the relative Cayley graph $\Gamma(G,X\cup\mathcal{H})$ is acylindrical.
\end{thm}

For the rest of this section, we fix a group $G$ which is  hyperbolic relative to a collection of subgroups $\{H_\lambda\}_{\lambda \in \Lambda}$ and let $X$ be a finite relative generating set of $G$. We will also assume that $G$ is a \emph{non-elementary relatively hyperbolic group}, which means that $G$ is not virtually cyclic and each $H_\lambda$ is a proper subgroup of $G$. Let $S=\Gamma(G,X\cup\mathcal{H})$ be the corresponding relative Cayley graph of $G$. We will identify $G$ with the set of vertices of $S$ and will always use $1 \in G$ as the basepoint.
By the previous two results, $S$ is a hyperbolic metric space and the action of $G$ on $S$ is acylindrical, hence all loxodromic elements of $G$ are WPD elements. We also note that an element $g\in G$ is loxodromic if and only if $g$ has infinite order and is not conjugate to an element of $H_\lambda$, for any $\lambda \in \Lambda$, by \cite[Corollary~4.20]{Osi06}.

Any subset $Y \subseteq G$ gives rise to the \emph{extended word metric} $\d_Y:G \times G \to \R\cup \{\infty\}$, where $\d_Y(g,h)$ is the length of a shortest word over
$Y^{\pm 1}$, representing the element $g^{-1}h$ in $G$, if $g^{-1}h \in \langle Y \rangle$, or $\d_Y(g,h)=\infty$ if $g^{-1}h \notin \langle Y \rangle $ in $G$.

Consider a path $p$ in $S$. Then the edges of $p$ are labeled by elements from $X \sqcup \mathcal{H}$. We will say that $p$ is \emph{tame} if the following condition holds. Suppose that $e_1$ and $e_2$ are two distinct edges of $p$ starting at vertices $g_1,g_2 \in G$ and labeled by elements from $H_\lambda$, for some $\lambda \in \Lambda$. Then $g_1$, $g_2$ belong to different left cosets of $H_\lambda$ in $G$, i.e., $g_2 \notin g_1 H_\lambda$.

Since any two vertices in the same $H_\lambda$-coset are at distance at most $1$ in $S$, we see that any geodesic path in $S$ is automatically tame. It is also obvious that any path labeled by a word in $X$ is tame.

\begin{lem}\label{lem:rh-dY}
Using the above notation, there exists a finite subset $Y \subseteq G$ such that the following are true for the extended word metric $\d_Y(\cdot,\cdot)$ on $G$.
\begin{itemize}
  \item[(i)] For any $\lambda \ge 1$ and $c \ge 0$ there exists $\xi \ge 0$ such that if $p$ and $q$ are $(\lambda,c)$-quasi-geodesics in $S$, $p$ is tame and  $p_-=q_-$, $p_+=q_+$ then for any vertex $x$ of $p$ there is a vertex $y$ of $q$ satisfying $\d_Y(x,y)\le \xi$.
   \item[(ii)] There is a constant $\nu \ge 0$ such that given any geodesic triangle in $S$ with sides $p$, $q$, $r$, for any vertex $y \in p$
      there is a vertex $z \in q \cup r$ satisfying $\d_Y(y,z) \le \nu$.
\end{itemize}
\end{lem}

\begin{proof} 
When $G$ is finitely generated, (ii) is given by \cite[Theorem 3.26]{Osi06}. The same proof works in general (cf. \cite[Lemma 5.3]{Osi13}). 

Similarly, when $G$ is finitely generated, (i) is given by \cite[Proposition~3.15]{Osi06}. Again the proof of \cite[Proposition~3.15]{Osi06} works in the general case without any significant changes.
\end{proof}

\begin{rem}\label{rem:quadr-2nu-slim} Claim (ii) of Lemma \ref{lem:rh-dY} implies that if $Q$ is a geodesic quadrangle in $S$ and $q$ is a side of $Q$ and
then any vertex of $q$ is contained in the $2\nu$-neighborhood of the union of the other sides of $Q$ with respect to the extended word metric $\d_Y(\cdot,\cdot)$.
\end{rem}

There are several equivalent definitions of relatively quasi-convex subgroups, see \cite{Hru}. We will use the one from \cite[Definition 6.10]{Hru}.

\begin{defn}\label{def:rel_qc_sbgp} We will say that a subgroup $H \leqslant G$ is \emph{relatively quasi-convex} if there exists a finite subset $Z \subseteq G$
and a constant $\eta \ge 0$ such that the following is satisfied.
If $r$ is a geodesic segment in $S=\Gamma(G,X\cup\mathcal{H})$ with $r_-,r_+ \in H$, then for any vertex $z$ of $r$ there
is $h \in H$ such that $\d_Z(z,h) \le \eta$.
\end{defn}

\begin{prop} \label{prop:rel_qc-rwpd} Let $H \leqslant G$ be a relatively quasi-convex subgroup. Then every loxodromic element $g \in G$
is $H$-\apt{} (with respect to the natural action of $G$ on $\Gamma(G,X\cup\mathcal{H})$).
\end{prop}

\begin{proof} Let $g \in G$ be a loxodromic element with respect to the natural action of $G$ on $S=\Gamma(G,X\cup\mathcal{H})$.
Let $\d(\cdot,\cdot)$ denote that standard path metric on $S$. Suppose that $Z \subseteq G$ and $\eta \ge 0$ are the finite subset and the constant from Definition \ref{def:rel_qc_sbgp},
and $Y \subseteq G$ is the finite subset provided by Lemma \ref{lem:rh-dY}. Without loss of generality, we can further assume that
$g \in X$ and $Y,Z \subseteq X$, as relative hyperbolicity of $G$ is independent of the  choice of a finite relative generating set (see \cite[Theorem~2.34]{Osi06}).

In view of Remark~\ref{rem:apt-indep_of_basepoint}, it is enough to check that the statement of Definition~\ref{defn:rwpd} holds for $s=1 \in S$.
Since the element $g \in X$ is loxodromic, there exist $\lambda \ge 1$ and $c \ge 0$ such that any path labelled by a power of $g$ in $S$ is $(\lambda,c)$-quasi-geodesic.
Let $\xi,\nu \ge 0$ be the constants given by Lemma~\ref{lem:rh-dY}.

For any given $C \ge 0$, we can choose $L \in \N$ so that $\d(1,g^L) > \xi+2\nu+C$, and suppose that $u \in G$ satisfies $\d(1,u) \le C$ and $\d(ug^{2L},h) \le C$, for some
$h \in H$.

Let $Q$ be a geodesic quadrangle in $S$ with vertices $1$, $u$, $ug^{2L}$ and $h$, and let $p$ be the $(\lambda,c)$-quasi-geodesic path in $S$ starting at $u$ and
labelled by $g^{2L}$. Then $p$ is tame and $x=ug^L$ is a vertex of $p$, and so, by Lemma~\ref{lem:rh-dY}.(i), $\d_Y(x,y) \le \xi$ for some vertex $y$ on the side $[u,ug^{2L}]$ of $Q$.

Since $Y \subseteq X$, it follows that $\d(x,y) \le \d_X(x,y) \le \d_Y(x,y) \le \xi$, and so
\[\d(y,u) \ge \d(u,x)-\d(y,x)\ge \d(u,ug^L)-\xi=\d(1,g^L)- \xi>2\nu+C. \]
Similarly, $\d(y,ug^{2L})> 2\nu+C$.
Therefore $y$ cannot be contained within $2\nu$-neighborhoods of the sides $[1,u]$ and $[ug^{2L},h]$, of $Q$ in $S$, so, in view of Remark \ref{rem:quadr-2nu-slim}
and the fact that $Y \subseteq X$, we can conclude that there must exist a vertex $z$ on the side $[1,h]$ such that $\d_Y(y,z) \le 2\nu$.

Finally, by the relative quasi-convexity of $H$, we know that $\d_Z(z,h_1) \le \eta$ for some $h_1 \in H$. Recalling that $Y\cup Z\subseteq X$, we deduce that
\begin{equation}\label{eq:d_X}
\d_X(x,h_1)\le \d_X(x,z)+\d_X(z,h_1)\le \d_Y(x,z)+\d_Z(z,h_1)\le \xi+2\nu+\eta.
\end{equation}

Let $U=\{f \in G \mid \d_X(1,f) \le \xi+2\nu+\eta\}$. Then $U$ is a finite subset of $G$, because $X \subseteq G$ is finite, and \eqref{eq:d_X}
implies that $h_1^{-1}x=h_1^{-1}ug^L \in U$. Thus $u  \in h_1Ug^{-L} \subseteq HU'$, where $U'$ is the finite subset of $G$ defined by $U'=Ug^{-L}$.
Therefore \[\{u \in G \mid \d(1,u) \le C, \d(ug^{2L},H) \le C\} \subseteq HU',\]
i.e., $g$ is $H$-\apt{}. Thus the proposition is proved.
\end{proof}

Combining Proposition \ref{prop:rel_qc-rwpd} and Corollary~\ref{cor:ex_of_trans}, we obtain the following.

\begin{cor}\label{cor:rh_trans}
Suppose that $H_1,\dots,H_k$ are relatively quasi-convex subgroups of $G$ and $F \leqslant G$ contains a loxodromic element. If $|F:(F \cap vH_iv^{-1})|=\infty$, for all $v \in G$ and $i =1,\dots,k$, then $F$ contains a loxodromic element $f$ which is transverse to $H_i$, for each $i=1,\dots,k$. In particular, if each $H_i$ has infinite index in $G$, then $G$ contains a loxodromic element $f$ which is transverse to $H_i$, for each $i=1,\dots,k$. 
\end{cor}

Next we show that an analogue of Proposition \ref{p:tt} holds in the relatively quasi-convex setting. The argument is very similar to the proof of Proposition~\ref{p:tt}, the only difference is that we have to deal with relative quasi-convexity.

\begin{prop}\label{prop:relqctt}
Let $G$ be a non-elementary relatively hyperbolic group with $E(G)=\{1\}$.  Let $H$ and $K$ be infinite index relatively quasi-convex subgroups of $G$ and let $\mathcal{F}$ be a finite subset of $G$. Suppose that $\mu$ is a permissible probability measure on $G$ with generating support and let $w_n$ be a random walk with respect to $\mu$. Then the subgroup $L_n=\langle w_n^{-1}Hw_n, K\rangle$ will satisfy the following conditions asymptotically almost surely.
\begin{enumerate}
\item[(a)] $L_n$ is  relatively quasi-convex in $G$;
\item[(b)] $L_n\cap \mathcal F=K\cap \mathcal F$ and $w_nL_nw_n^{-1}\cap \mathcal F=H\cap \mathcal F$;
\item[(c)] $L_n$  has infinite index in $G$.
\end{enumerate}
\end{prop}

\begin{proof} 
%In view of Proposition~\ref{prop:rel_qc-rwpd}, the argument is very similar to the proof of the quasi-convexity of the orbit of $L_n$ in part (a) of Proposition~\ref{p:tt}. The only difference is that we use Lemma~\ref{lem:rh-dY}.(i) instead of Lemma~\ref{lem:stab_of_qgeod}.

According to Definition~\ref{def:rel_qc_sbgp}, there exist finite subsets $Z_H$ and $Z_K$ of $G$ and a constant $\eta\geq 0$ such that the following holds. If $r$ is a geodesic segment in $S$ with $r_-,r_+ \in H$, then for any vertex $z$ of $r$ there is $h \in H$ such that $\d_{Z_H}(z,h) \le \eta$. Also if $r$ is a geodesic segment in $S$ with $r_-,r_+ \in K$, then for any vertex $z$ of $r$ there is $k \in K$ such that $\d_{Z_K}(z,k) \le \eta$. Let $Y$ be the finite subset of $G$ given by Lemma \ref{lem:rh-dY}. Fix geodesics $[1, w_n]$ and $[1, w_n^{-1}]$  in $S$, and whenever we have a path labeled by $w_n^{\pm 1}$ we will use the appropriate translate of one of those fixed geodesics. Let $Z_n=Z_H\cup Z_K\cup Y\cup\{v^{-1}w_n\;|\; v\text{ is a vertex on } [1, w_n]\}\cup\{v\;|\; v \text{ is a vertex on } [1, w_n^{-1}]\} $; we will show that $L_n$ satisfies the definition of relative quasi-convexity with respect to the finite subset $Z_n$ asymptotically almost surely.

By Corollary~\ref{cor:rh_trans} there exists a loxodromic element $h \in G$ which is transverse to $H$ and $K$.  Set $V=(H\cup K)\setminus\{1\}$; note that $V$ is quasi-convex in $S$, by Lemma~\ref{lem:union_of_qc}, and $h$ is transverse to $V$, by Lemma~\ref{lem:transuc}. According to Theorem~\ref{thm:rh_acy}, $h$ is a WPD element since the action of $G$ on $S$ is acylindrical. By Theorem~\ref{thm:rw}, we can assume that if $W$ is a word in the alphabet $V\cup\{w_n^{\pm 1}\}$ with no consecutive letters belonging to $V$, then any path in $S$ labeled by $W$ is an $(8, c_n)$-quasi-geodesic.

 Let $p$ be a geodesic path in $S$ with $p_-, p_+\in L_n$. Applying a translation by an element of $L_n$, we can assume that $p$ starts at $1$ and terminates at some $g \in L_n$. Let $W$ be a word as in \eqref{eq:word_for_g}, representing $g$ in $G$ and considered as a reduced word over the alphabet $V \cup \{w_n^{\pm 1}\}$. Let $q$ be the path in $S$ from $1$ to $g$ which is labeled by $W$. Then $q$ is $(8, c_n)$-quasi-geodesic, so, by Lemma~\ref{lem:rh-dY}, there is a constant $\varkappa_n \ge 0$ such that for every point $x$ of $p$ there is a point $y$ of $q$ such that $\d_{Z_n}(x, y)\le\d_Y(x,y) \le \varkappa_n$. By construction, the path $q$ is a concatenation of geodesic subpaths labeled by $w_n^{\pm 1}$, $k_i$  or $h_j$, and $y$ must belong to one of them. 

Suppose, first, that $y$ lies on a subpath $r$, of $q$, labeled by $k_i$, for some $i=0,\dots,m$:
\[r=[k_0 w_n^{-1} h_1 {w_n} k_1 \dots w_n^{-1} h_i {w_n}, k_0 w_n^{-1} h_1 {w_n} k_1 \dots w_n^{-1} h_i {w_n}k_{i}]=u([1,k_i]),\]
where $u=k_0 w_n^{-1} h_1 {w_n} k_1 \dots w_n^{-1} h_i {w_n} \in L_n$. Then $u^{-1}y\in [1,k_i]$, so by the relative quasi-convexity of $K$ we have that $\d_{Z_n}(u^{-1}y, K)\le \d_{Z_K}(u^{-1}y, K)\le \eta$. Hence
\[\d_{Z_n}(y, L_n)\le \d_{Z_n}(y, uK)=\d_{Z_n}(u^{-1}y, K)\le \eta.\]

If $y$ lies on a subpath labeled by $w_n$ then $y=uv$ where $v$ is a vertex on $[1, w_n]$ and $u$ is an element of the form $k_0 w_n^{-1} h_1 {w_n} k_1...w_n^{-1} h_i$. In this case, $yv^{-1}w_n=uw_n\in L_n$ and $v^{-1}w_n\in Z_n$, by the definition of $Z_n$, so 
\[
\d_{Z_n}(y, L_n)\leq \d_{Z_n}(y, yv^{-1}w_n)\le 1.
\]

If $y$ belongs to a subpath of $q$ labeled by $w_n^{-1}$, then $y=uv$ where $v$ is a vertex on $[1, w_n^{-1}]$ and $u$ is an element of the form $k_0 w_n^{-1} h_1 {w_n} k_1...w_n^{-1} h_iw_nk_i$. Then $u\in L_n$ and $v\in Z_n$, so
\[
\d_{Z_n}(y, L_n)\leq \d_{Z_n}(y, u)=\d_{Z_n}(uv, u)\le 1.
\]

The final possibility is that $y$ belongs to a geodesic subpath $r$, of $q$, labeled by $h_j$, for some $j=1,\dots,m$, so that
\[r=u([1,h_j]), \text{ where } u=k_0 w_n^{-1} h_1 {w_n} k_1\dots k_{j-1}w_n^{-1}.\]  

Then $u^{-1}y\in [1,h_j]$, so $\d_{Z_n}( u^{-1}y, H)\le \d_{Z_H}(u^{-1}y, H)\le \eta$, by the relative quasi-convexity of $H$. Note that $w_n\in Z_n$ by construction, so for all $h\in H$, $\d_{Z_n}(h, hw_n)\le 1$. Hence we have
\[\d_{Z_n}(y, L_n)\le \d_{Z_n}(y, uHw_n)=\d_{Z_n}(u^{-1}y, Hw_n)\le \eta+1.\]
Thus we can conclude that $\d_{Z_n}(x, L_n)\le \varkappa_n+\eta+1$, which implies that $L_n$ is relatively quasi-convex.

The proofs of parts (b) and (c) are identical to the proofs of the corresponding parts of Proposition~\ref{p:tt}.
\end{proof}

Theorem~\ref{Thm:main2} from the Introduction follows from the following statement, whose proof is identical to the proof of Theorem~\ref{Thm:CC}, with Proposition~\ref{p:tt}  replaced by Proposition~\ref{prop:relqctt}.

\begin{thm}\label{thm:relhyp_CC}
Let $G$ be a non-elementary relatively hyperbolic group with $E(G)=\{1\}$. Let  $\mu\in \Prob(G)$ be a permissible probability measure on $G$ with generating support. Then the action of $G$ on the closure of the space of infinite index relatively quasi-convex subgroups in $\Sub(G)$ is topologically $\mu$-mixing.
\end{thm}

Corollary~\ref{Thm:main2} from the Introduction follows from Theorem~\ref{thm:relhyp_CC} and the following fact.

\begin{lem}\label{lem:inf_many_ends->rh} Let $G$ be a finitely generated group with infinitely many ends. Then $G$ is non-elementary relatively hyperbolic and every finitely generated subgroup of $G$ is relatively quasi-convex.
\end{lem}

\begin{proof} By a famous result of Stallings \cite{Sta}, $G$ decomposes either as a non-trivial amalgamated free product $A*_C B$, where $A,B \leqslant G$ and $C=A \cap B$ is finite, or as an HNN-extension $A*_{C^t=D}$, where $A \leqslant G$ and $C, D$ are isomorphic finite subgroups of $A$. Let $T$ be the Bass-Serre tree corresponding to this splitting of $G$. Then $G$ acts on $T$ by isometries, without edge inversions, with finite edge stabilisers and with finitely many orbits of vertices and edges. 

Evidently, $T$ is a fine hyperbolic graph. Therefore $G$ is hyperbolic relative to any collection of representatives of vertex stabilizers (e.g., $\{A,B\}$ in the case when $G \cong A*_C B$ or $\{A\}$ in the case when $G \cong A*_{C^t=B}$) by Definition \ref{defn:relhyp}. 
Note that $G$ is non-elementary, as it is not virtually cyclic and the vertex stabilizers are proper subgroups of $G$. 

Let $H$ be a finitely generated subgroup of $G$. Then there is a convex subtree $T_1$ of $T$ which is $H$-invariant and on which $H$ acts cocompactly (see \cite[Proposition~I.4.13]{DD}). It now follows from the work of Mart\'{\i}nez-Pedroza and Wise \cite{MPW} that $H$ is relatively quasi-convex in $G$.
\end{proof}

\end{document}